\documentclass{article}

\usepackage{amsmath,amssymb,amsthm,graphicx,verbatim,tabularx,MnSymbol}
\usepackage{algorithmic,algorithm,newfloat,color,subcaption}
\usepackage[hidelinks]{hyperref}

\hypersetup{
    colorlinks=false,
    linkcolor=black,
    filecolor=black,      
    urlcolor=black,
    pdfpagemode=FullScreen,
    }

\newtheorem{theorem}{Theorem}
\newtheorem{lemma}{Lemma}
\newtheorem{corollary}{Corollary}

\def\mR{{\mathbb R}}
\def\bP{\mathbf{P}}
\def\cL{{\cal L}}
\def\cB{{\cal B}}

\def\pd{\succ}
\def\psd{\succeq}
\def\nd{\prec}
\def\nsd{\preceq}
\def\xh{\hat{x}}
\def\yh{\hat{y}}
\def\sh{\hat{s}}
\def\xb{\bar{x}}
\def\yb{\bar{y}}
\def\sb{\bar{s}}

\begin{document}

\title{Central Path Art}

\author{Thor Catteau$^{\,a1}$,
    Benjamin Glancy$^{\, a2}$,
    Allen Holder$^{\, a3^*}$, 
    Angela Milkowski$^{\, a4}$, \\[4pt]
    Alexa Renner$^{\, a5}$,
    Connor Tasik$^{\, a6}$, and
    Rebecca Testa$^{\, a7}$
    \\ \\
\parbox{.9\textwidth}{\footnotesize
$^a\,$Department of Mathematics, Rose-Hulman Institute of Technology, \\
        \hspace*{6pt}Terre Haute, IN, USA} \\[8pt]
\parbox{.9\textwidth}{\footnotesize
        \hspace*{-10pt}
        \begin{tabular}{ll}
        $^1\,$catteae@rose-hulman.edu, &
        $^2\,$glancybc@rose-hulman.edu \\[2pt]
        $^3\,$holder@rose-hulman.edu, &
        $^4\,$milkowaj@rose-hulman.edu \\[2pt]
        $^5\,$renneram@rose-hulman.edu, &
        $^6\,$tasikca@rose-hulman.edu \\[2pt]
        $^7\,$testarl@rose-hulman.edu 
        \end{tabular} \\
        }  \\[10pt]
\parbox{0.9\textwidth}{\footnotesize
$^*$Corresponding author} \\[0pt]
}

\maketitle

\begin{abstract}

The central path revolutionized the study of
optimization in the 1980s and 1990s due to its favorable
convergence properties, and as such, it has been
investigated analytically, algorithmically, and 
computationally. Past pursuits have primarily focused
on linking iterative approximation algorithms to the
central path in the design of efficient algorithms
to solve large, and sometimes novel, optimization problems. 
This algorithmic intent has meant that the central
path has rarely been celebrated as an aesthetic entity in 
low dimensions, with the only meager exceptions being illustrative
examples in textbooks. We undertake this low dimensional
investigation and illustrate the artistic use of the central path to
create aesthetic tilings and flower-like constructs in two and three 
dimensions, an endeavor that combines mathematical rigor and
artistic sensibilities. The result is a fanciful and enticing
collection of patterns that, beyond computer generated images, 
supports math-aesthetic designs for novelties and museum-quality 
pieces of art. \\\


\end{abstract}

\section{Historical Basis and Motivation} \label{sec-intro}

The concepts of an analytic center and an interior point
algorithm began with Frisch and Huard at the inception of the 
modern age of optimization and operations research~\cite{frisch1955,huard64,huard67},
see also the original work of Dikin~\cite{dikin67} and
Fiacco and McCormick~\cite{fiacco68}. These pioneering
articles motivated the idea of iterating through points interior 
to the feasible region toward an optimal solution, which differed 
from the prevailing vertex approach of the simplex method. However, the 
simplex algorithm left open the question of whether or not the class of 
linear programming problems was solvable in polynomial time~\cite{klee72}. 
Leonid Khachiyan famously answered this question in 1979 by presenting 
a polynomial time algorithm for linear programs~\cite{khachiyan79}, but
Khachiyan's algorithm proved impractical and the simplex method continued to 
prevail until Narendra Karmarkar introduced a more successful polynomial 
time algorithm in 1984~\cite{karmarkar84b,karmarkar84a}. Although not 
obvious at the time of their publications, the algorithms of both Khachiyan 
and Karmarkar were interior-point methods related to the original works of 
Huard, Dikin, and Fiacco and McCormick, a connection that was noticed soon 
thereafter in~\cite{gill86}. The practical and mathematical success of 
Karmarkar's algorithm marked the advent of interior-point methods, and it
led to thousands upon thousands of publications that have revolutionized the 
field of optimization~\cite{forsgren02,wright98,wright05}. An overriding
conclusion from this literature is that efficient and efficacious 
solution procedures to a wide class of salient problems rely on the 
favorable convergence properties of the central path, and hence,
the central path is one of the most important mathematical entities of the 
twentieth century.

The central path is a parametrization of the positive real half-line in a convex
set, and each element of the path is an analytic center. The pertinence 
of the central path to effective algorithm design has prompted numerous 
mathematical studies to understand its analytic, geometric, and topological 
properties. Sonnevend undertook many of the original 
studies~\cite{sonnevend86,sonnevend88}, see also~\cite{jarre88}, but the 
mathematical pursuits continued for decades as briefly illustrated 
by~\cite{caron02,deza06,halicka99,holder00,deloera12,nematollahi08,vavasis96}.
We refer readers to see~\cite{terlaky2009} for a succinct historical
review of interior point methods.

Our motivation to create art with the central path stems from its 
mathematical and computational properties, which again, have revolutionized the 
field of optimization. Many have generated paths in low dimensions to limn the 
behavior of interior-point methods, and in doing so, they have certainly
found it easy, if not alluring, to be captivated by an esoteric and 
quizzical elegance. We leverage this elegance to create aesthetic images 
in two and three dimensions, and we use these images to create artistic
wall hangings, works of stained glass, backpack tags, beverage coasters, 
holiday ornaments, and three-dimensional sculptures. A continued pursuit is
to assemble three-dimensional sculptures into an optimization garden full 
of (random) interior bouquets that dance and sparkle.

We straightforwardly introduce the central path and its convergence analysis
in Sections~\ref{sec-acp} and~\ref{sec-conv} so
that it is reasonably accessible to undergraduates, although many
will require modest educational additions. For instance, 
few undergraduates have dedicated coursework in optimization 
and/or convex analysis, and standard courses in real analysis often forgo 
the Implicit Function Theorem. We employ such results and expect readers to 
undertake brief studies as needed. We also encourage the use of the 
{\em Mathematical Programming Glossary} to clarify definitions~\cite{mpg}.

We present two- and three-dimensional works of art in 
Sections~\ref{sec-2Dart} and~\ref{sec-3Dart}. Our projects separate
into two categories, one based on tilings of $k$-gons and another
based on floral facsimiles, e.g. daisies, thistles, tulips, roses, and trumpet 
flowers. The majority of our work to date stems from central paths in two-dimensions,
but we also demonstrate the use of Platonic solids. Some of our designs have
motivated new mathematical relationships --- so the quest to create art has lead to
a bit of new mathematics. We conclude in Section~\ref{sec-conclusion} with a succinct 
review and a discussion of our future goals.

\section{The Analytic Central Path} \label{sec-acp}

The specific model from which we proceed is important because the
central path depends on the algebraic description of an optimization
problem and not directly on its geometry, a fact first recognized by 
Sonnevend~\cite{sonnevend86} and then analyzed in~\cite{caron02}. 
We assume that
\[
G:\mR^n \rightarrow \mR^m: 
    x \mapsto \left( \begin{array}{c} g_1(x) \\ g_2(x) \\ \vdots \\ g_m(x)
        \end{array} \right)
\]
is twice smooth, that $\nabla G(x)$ has full column rank, and that 
each component function is convex. The assumed convexity of each $g_i(x)$ 
ensures that 
\[
\{ x \, : \, G(x) \le 0 \}
\]
is convex. We further assume that this set is compact and that its strict 
interior is nonempty, i.e. $\{ x \, | \, G(x) < 0\} \neq \emptyset$. The 
assumption of compactness is atypical but apt for our purposes.
Our mathematical model is
\begin{equation} \label{eq-mainProb}
\max \left\{ c^T x + \mu \sum_{i=1}^m \ln (s_i) \, 
    : \, G(x) + s = 0, \, s \ge 0 \right\},
\end{equation}
with $c \in \mR^n$ defining a linear term in the objective function
and $\mu$ being a positive parameter. The objective function maps into
the extended reals, $\overline{\mR} = \mR \cup \{ -\infty, \infty\}$,
with the objective value being $-\infty$ if some $s_i = 0$.

Our first result shows that~\eqref{eq-mainProb} is well-posed.
\begin{theorem}
The optimization problem in~\eqref{eq-mainProb} is well-posed
for $\mu > 0$ under the stated assumptions of $G$.
\end{theorem}
\begin{proof}
We have feasibility by assumption, and the objective function
is continuous over the strict interior of the feasible region.
Moreover, if $(x^k, s^k) \rightarrow (\xh, \sh)$ is a convergent 
sequence of feasible elements, then
\[
\lim\!\sup_{\hspace*{-14pt}k \rightarrow \infty} 
    \left( c^T x^k + \mu \sum_{i=1}^m \ln (s^k_i) \right)
    \le \left( c^T \xh + \mu \sum_{i=1}^m \ln (\sh_i) \right).
\]
This inequality follows if $\sh > 0$ because the objective function 
is then continuous at $(\xh, \sh)$, and consequently, the limit 
supremum is finite and we satisfy the inequality as an equality. 
If $\sh_i = 0$ for some $i$, then any subsequence of $(x^k, s^k)$, 
say $(x^{k_j}, s^{k_j})$, satisfies $s_i^{k_j} \rightarrow 0$ as 
$j \rightarrow \infty$, and hence, the left-hand side is $-\infty$, 
which matches the right-hand in this case. We conclude that the 
objective is upper semicontinuous, which establishes the result 
because upper semicontinuous functions attain their maximums over 
compact sets.
\end{proof}

We assume for notational convenience that capitalized vectors
represent diagonal matrices whose main diagonals are the elements of 
the vector, so if $x$ is an $n$-vector, then $X$ is an $n \times n$ diagonal
matrix such that $X_{ii} = x_i$. We further use $A \pd 0$ ($A \nd 0$)
and $A \psd 0$ ($A \nsd 0$) to indicate respectively that $A$ is either 
positive (negative) definite or positive (negative) semidefinite. 
Letting $f(x,s)$ be the objective function in~\eqref{eq-mainProb}, 
we have with these notational conventions that the Hessian of $f(x,s)$
at a strictly feasible point $(x,s)$ satisfies
\[
\nabla^2 f(x,s) = 
    \left[ \begin{array}{cc} 0 & 0 \\[4pt] 0 & -S^{-2} \end{array} \right]
    \nsd 0.
\]
So the objective function is concave over the strict interior of the
feasible set. If any component of $s^1$ or $s^2$ is zero for the feasible
elements $(x^1, s^1)$ and $(x^2, s^2)$, then we also have
\[
f\left(  (1 - \theta)x^1 + \theta x^2, 
    (1 - \theta) s^1 + \theta s^2 \right) \ge 
        (1 - \theta) f(x^1,s^1) + \theta f(x^2, s^2) = -\infty.
\]
So the objective function is 
concave, and subsequently, the optimization problem in~\eqref{eq-mainProb} 
is convex.

The Lagrangian of~\eqref{eq-mainProb} is
\[
\cL(x,s,y,\sigma) = c^Tx + \mu \sum_{i=1}^m \ln(s_i)
                            - y^T(G(x) + s) - \sigma^T s,
\]
and the convexity of the problem means that the first-order
Lagrange conditions are both necessary and sufficient for
optimality. So $(x,s)$ is optimal if and only if there
are $m$-vectors $y$ and $\sigma$ so that,
\[
\begin{array}{rcl}
\displaystyle
\nabla_{x,s} \cL(x,s,y,\sigma)
    & \hspace*{-8pt}=\hspace*{-8pt} & \left( \begin{array}{c} c - \nabla G(x)^T y \\[4pt]
            \mu \, S^{-1} e - y - \sigma \end{array} \right)
    = \left( \begin{array}{c} 0 \\ 0 \end{array} \right), \\[18pt]
\displaystyle G(x) + s & \hspace*{-8pt}=\hspace*{-8pt} & 0, \; \sigma^T s = 0, \; s \ge 0, 
    \mbox{ and } \sigma \ge 0,
\end{array}
\]
with $e$ being the vector of ones and length being decided by
the context of its use. Note that $s$ has positive components
in an optimal solution, and hence, we know that $s > 0$, 
$\sigma \ge 0$, and $\sigma^T s = 0$, from which we gain that 
$\sigma = 0$. Also note that $\mu S^{-1} e - y = 0$ ensures that
$y > 0$ because $\mu > 0$. So if we re-express $\mu S^{-1} e - y = 0$ as
$Sy - \mu e = 0$, then the necessary and sufficient conditions reduce to,
\begin{equation} \label{eq-reducedLagrangeCond}
\left.
\begin{array}{rcl}
\nabla G(x)^T y - c & = & 0, \; y > 0, \\[4pt]
G(x) + s & = & 0, \; s > 0, \mbox{ and} \\[4pt]
Sy - \mu e & = & 0.
\end{array}
\right\}
\end{equation}
We define $F$ to be
\[
F: \mR^{2m+n+1} \rightarrow \mR^{2m+n}
    : (x,s,y,\mu) \mapsto \left(
        \begin{array}{c} 
            \displaystyle \nabla G(x)^T y - c \\[4pt]
            \displaystyle G(x) + s \\[4pt]
            Sy - \mu e 
        \end{array}
        \right),
\]
so that the necessary and sufficient conditions for $\mu > 0$ 
are succinctly,
\begin{equation}\label{eq-necSuffCond}
F(x,s,y,\mu) = 0 \; \mbox{ and } \; (s,y) > 0.
\end{equation}
We now have
\[
\nabla_{x,s,y} \, F(x,s,y,\mu) = \left[ 
    \renewcommand{\arraystretch}{2}
    \renewcommand{\arraycolsep}{10pt}
    \begin{array}{ccc}
    \displaystyle \sum_{i=1}^m \, y_i \nabla^2 g_i(x) & 0 &
       \nabla G(x)^T \\
    \nabla G(x) & I & 0 \\
    0 & Y & S
    \end{array} \right],
\]
and we argue that this matrix is invertible, from which we gain 
that $x$ and $s$ are analytic functions of $\mu$ from the 
Implicit Function Theorem.
\begin{lemma} \label{lemma-invMat}
Assume that $Q$ and $D$ are $n \times n$ matrices and that
$B$ is an $m \times n$ matrix. Further assume that $Q \psd 0$,
$D \pd 0$, and that $B$ has full column rank. Then
\[
\left[ \begin{array}{cc}
    Q & B^T \\[6pt] DB & -I \end{array} \right]
\]
is nonsingular.
\end{lemma}
\begin{proof}
Observe that
\[
\left[ \begin{array}{cc}
    Q & B^T \\[6pt] DB & -I \end{array} \right]
\left( \begin{array}{c} w \\ u \end{array} \right)
=
\left( \begin{array}{c} 0 \\ 0 \end{array} \right)
\]
reduces to
\[
(Q + B^T D B) \, w = 0.
\]
The properties of $Q$, $B$, and $D$ ensure that
$Q + B^T D B \pd 0$, from which we conclude that
$w$, and subsequently $u$, are both zero.
\end{proof}
\begin{theorem} \label{thm-pathSmoothness}
There is a unique solution to~\eqref{eq-mainProb} for
each $\mu > 0$, say $x(\mu)$, $s(\mu)$ and $y(\mu)$, all of
which are twice-smooth functions of $\mu$.
\end{theorem}
\begin{proof}
We first observe that
\[
\left[ 
    \renewcommand{\arraystretch}{2}
    \renewcommand{\arraycolsep}{10pt}
    \begin{array}{ccc}
    \displaystyle \sum_{i=1}^m \, y_i \nabla^2 g_i(x) & 0 &
       \nabla G(x)^T \\
    \nabla G(x) & I & 0 \\
    0 & Y & S
    \end{array} \right]
    \left( \begin{array}{c} u \\ v \\ w \end{array} \right)
    = \left( \begin{array}{c} 0 \\ 0 \\ 0 \end{array} \right)
\]
gives $v = -\nabla G(x) u$ and $Yv + Sw = 0$. So we can replace
the bottom two equalities with 
\[
-Y \, \nabla G(x) \, u + Sw = 0, \mbox{ which is the same as }
    S^{-1} Y \, \nabla G(x) \, u - w = 0. 
\]
So the system reduces to
\begin{equation} \label{eq-reducedHessianSys}
\left[ 
    \renewcommand{\arraystretch}{2}
    \renewcommand{\arraycolsep}{10pt}
    \begin{array}{cc}
    \displaystyle \sum_{i=1}^m \, y_i \nabla^2 g_i(x) & 
        \nabla G(x)^T \\
    S^{-1}Y \nabla G(x) & -I 
    \end{array} \right]
    \left( \begin{array}{c} u \\ w \end{array} \right)
    = \left( \begin{array}{c} 0 \\ 0 \end{array} \right).
\end{equation}
We know that
\[
\sum_{i=1}^m \, y_i \nabla^2 g_i(x) \psd 0
\]
because each $g_i(x)$ is twice smooth and convex, and hence, each 
Hessian satisfies $\nabla^2 g_i(x) \psd 0$. Lemma~\ref{lemma-invMat}
now ensures that the matrix in~\eqref{eq-reducedHessianSys} is 
nonsingular with
\[
Q = \sum_{i=1}^m \, y_i \nabla^2 g_i(x) \psd 0, \;\;
B = \nabla G(x), \; \mbox{ and } \; D = S^{-1}Y \pd 0.
\]
So $u$ and $w$ are zero, and subsequently, so is $v$.
We conclude that $\nabla_{x,s} \, F(x,s,\mu)$ is nonsingular,
and the result follows from the Implicit Function Theorem.
\end{proof}
We comment that the functions $x(\mu)$, $s(\mu)$, and $y(\mu)$
inherit the analytic properties of $G(x)$ from the Implicit
Function Theorem. So these functions are, for instance,
analytic in the common situation that the component functions of 
$G(x)$ are analytic. We note further that $x$, $s$, and $y$ also 
depend on the data that defines $G(x)$ and $c$. So if $G(x)$ is
the affine transformation $G(x) = Ax - b$, then $x$, $s$, and $y$
depend on $\mu$ as well as the triple $(A,b,c)$, a fact that leads
to additional analyses, see~\cite{holder04,sonnevend86}.

Our artistic goals prompt us to define the central path as
\[
\bP(G(x), c) = \left\{ x(\mu) : \mu > 0 \right\},
\]
which, somewhat oddly, lacks an explicit dependence on $s(\mu)$ 
and/or $y(\mu)$. The elements of $s$ are slack variables in the language 
of optimization, and we have included them to simplify our calculations.
The elements of $y$ are Lagrange, or dual, variables in optimization,
and they give us a way to define optimality conditions algebraically. 
Explicitly including $s$ or $y$ as part of the central path places 
us in dimensions beyond visual appeal, and for this reason, our 
definition projects $\{ (x(\mu), \, s(\mu), y(\mu)) : \mu > 0 \}$ onto 
its $x$-components. We create art by controlling $G(x)$ and 
$c$, with each choice rendering a `brushstroke' in two or
three dimensions. We then assemble these brushstrokes into 
aesthetic ensembles.

\section{Convergence Analysis} \label{sec-conv}

We most commonly use the affine map $G(x) = Ax - b$ even though some of our
floral pieces more naturally align with $G(x)$ being a quadratic form. 
However, in these cases we show that an affine map still suffices, and
hence, the majority, and indeed all but one, of our models uses an affine map
computationally. The only computational quadratic case is one in which we can 
explicitly state the central path. Otherwise the affine map combines with our 
assumption of a bounded feasible region to give central paths in polytopes, 
an outcome that supports a meaningful and insightful convergence analysis. We 
specifically show that $x(\mu)$ converges as $\mu \uparrow \infty$ and as 
$\mu \downarrow 0$, and moreover, that these limits solve optimization problems 
that distinguish them as analytic centers. Our presentation is not new and is found,
for instance, in~\cite{roos2005}. 

We now assume that $G(x) = Ax - b$, from which we have that the necessary and sufficient 
conditions in~\eqref{eq-reducedLagrangeCond} become
\begin{equation} \label{eq-tradSys}
Ax + s = b, \; A^T y = c, \; Sy = \mu e, \; y > 0, \; s > 0.
\end{equation}
We first establish that solutions to this system are bounded if $s^T y$ is bounded.
\begin{theorem} \label{thm-boundedCP}
The set
\[
\cB = \{(x,y,s) : 
    Ax + s = b, \; A^T y = c, \; y > 0, \; s > 0, \; s^T y \le M \}
\]
is bounded for any $M > 0$.
\end{theorem}
\begin{proof}
Select $M > 0$ and $(\xh, \, \yh, \, \sh) \in \cB$. We then have for 
any other $(x,y,s) \in \cB$ that $\sh - s \in \mbox{col}(A)$ and 
$\yh - y \in \mbox{null}(A^T)$, and hence,
\[
0 = (\sh - s)^T(\yh - y) = \sh^T\yh - s^T\yh - \sh^Ty + s^Ty.
\]
So for any index $i$, we have
\[
\sh_i \, y_i \le \sh^T y + s^T \yh = \sh^T \yh + s^T y \le \sh^T \yh + M,
\]
and subsequently,
\[
0 \le y_i \le \frac{\sh^T \yh + M}{\sh_i}.
\]
An analogous argument shows that for any index $i$, 
\[
0 \le s_i \le \frac{\sh^T \yh + M}{\yh_i},
\]
which completes the proof because we also know that $x$ is 
bounded by assumption.
\end{proof}
\noindent We have the following corollary from the fact that 
$Sy = \mu e$ implies $s^T y = m \mu$.
\begin{corollary} \label{cor-boundedLevelSet}
If $\mu^k \downarrow 0$, then the solutions to~\eqref{eq-tradSys} 
are bounded.
\end{corollary}

We now show that $x(\mu)$ converges as $\mu \downarrow 0$ and as 
$\mu \uparrow \infty$, with the former limit solving the linear program,
\[
\max \{ c^T x \, : \, Ax + s = b, \, s \ge 0 \}.
\]
The fact that we can solve a linear optimization problem by following
the central path as $\mu \downarrow 0$ is the historical reason for the
path's importance, and in particular, following the central path to a
solution leads to polynomial time algorithms. Our artistic pursuits do
not rely on this historical importance, but the convergence properties
help us understand the `starting' and `ending' points of the paths 
used to create images. Both limits are analytic centers.

Our arguments rely on the concept of a support set, which for vector 
$v$ is
\[
\sigma(v) = \{ i : v_i \neq 0 \}.
\]
So the support set of $v$ is an index set containing the locations
at which $v$ is non-zero. We extend this notation in two ways. First,
the complement of $\sigma(v)$ is 
$\neg \sigma(v) = \{1, 2, \ldots, n\} \setminus \sigma(v)$ under the
assumption that $v$ is of length $n$. Second, $v_{\sigma(v)}$ is the
subvector of $v$ containing only the non-zero elements of $v$ (with order
preserved), and similarly, $v_{\neg \sigma(v)}$ is a vector of zeros. 
\begin{theorem} \label{thm-converge}
The following limits exist,
\[
\lim_{\mu \downarrow 0} \, (x(\mu), s(\mu)) = (x^*, \, s^*) \;\; \mbox{ and } \;\;
\lim_{\mu \uparrow \infty} \, (x(\mu), s(\mu)) = (x^c, \, s^c),
\]
and the first of these solves $\max \{ c^T x \, : \, Ax + s = b, \, s \ge 0 \}$.
We further have that $(x^*, s^*)$ is the unique solution to
\begin{equation} \label{eq-primalOptAC}
\max \left\{ \sum_{i \in \sigma(s^*)} \ln(s_i) \, : 
        Ax + s = b, \, s_{\sigma(s^*)} > 0, \, s_{\neg\sigma(s^*)} = 0, \, 
        c^T x = c^T x^* \right\}, 
\end{equation}
and that $(x^c, s^c)$ is the unique solution to
\begin{equation} \label{eq-primalAc}
\max \left\{ \sum_{i=1}^{m} \ln(s_i) \, : \, Ax + s = b, \, s > 0 \right\}.
\end{equation}
\end{theorem}
\begin{proof}
Assume $\mu^k \downarrow 0$ and set 
$(x^k, y^k, s^k) = (x(\mu^k), \, y(\mu^k), \, s(\mu^k) )$.
This sequence is bounded by Corollary~\ref{cor-boundedLevelSet}, 
and hence, it has a cluster point, say $(x^*, y^*, s^*)$. Without loss 
of generality, we assume $(x^k, y^k, s^k) \rightarrow (x^*, y^*, s^*)$ 
as $k \rightarrow \infty$. We have from the fact that $(x^k, y^k, s^k)$ 
satisfies~\eqref{eq-tradSys} that
\[
Ax^* + s^* = b, \; A^T y^* = c, \; s^* \ge 0, \;  y^* \ge 0, \;
    \mbox{ and } \; (s^*)^T x^* = 0.
\]
These are the necessary and sufficient conditions showing 
that $(x^*, s^*)$ and $y^*$ respectively solve
\begin{equation} \label{eq-lpOpt}
\max \{ c^T x \, : \, Ax + s = b, \, s \ge 0 \}
\;\; \mbox{ and } \;\;
\min \{ b^T y \, : \, A^T y = c, \; y \ge 0 \}.
\end{equation}
Notice that $s^k - s^* \in \mbox{col}(A)$ and $y^k - y^* \in \mbox{null}(A^T)$,
and hence, from the fact that $s_i^k y_i^k = \mu^k$ for any index $i$, we have
\begin{eqnarray*}
0 & = & (s^k - s^*)^T (y^k - y^*) \\[4pt]
  & = & (s^k)^T y^k - (s^*)^T y^k - (s^k)^T y^* + (s^*)^T y^* \\[4pt]
  & = & m \mu^k - (s^*)^T y^k - (s^k)^T y^*.
\end{eqnarray*}
We subsequently have
\[
m = \sum_{i \in \sigma(s^*)} \frac{s^*_i}{s^k_i} 
    \; + \sum_{i \in \sigma(y^*)} \frac{y^*_i}{y^k_i},
\]
from which we gain that $\sigma(s^*)$ and $\sigma(y^*)$ partition
$\{1, 2, \ldots, m\}$. So $(x^*, y^*, s^*)$ is a strictly
complementary solution to the primal-dual pair in~\eqref{eq-lpOpt},
from which we know that if $(\xb, \, \yb, \, \sb)$ is an optimal 
solution to~\eqref{eq-lpOpt}, then we are guaranteed to have 
$\sigma(\sb) \subseteq \sigma(s^*)$ and
$\sigma(\yb) \subseteq \sigma(y^*)$.

To see that $(x^*, s^*)$ uniquely solves~\eqref{eq-primalOptAC}, let
$(\xb, \, \yb, \, \sb)$ be an arbitrary optimal solution 
to~\eqref{eq-lpOpt}. Then an analogous argument to that above shows 
that
\[
m = \sum_{i \in \sigma(s^*)} \frac{\sb_i}{s^k_i} 
    \; + \sum_{i \in \sigma(y^*)} \frac{\yb_i}{y^k_i},
\]
and the arithmetic-geometric mean inequality gives
\[
\prod_{i \in \sigma(s^*)} \left( \frac{\sb_i}{s^k_i} \right)
    \; \prod_{i \in \sigma(y^*)} \left( \frac{\yb_i}{y^k_i} \right)
\le \; \frac{1}{m} \left(
    \sum_{i \in \sigma(s^*)} \frac{\sb_i}{s^k_i} 
    \; + \sum_{i \in \sigma(y^*)} \frac{\yb_i}{y^k_i} \right) = 1.
\]
One choice for $\yb$ is $y^*$, and in this case we have
\[
\prod_{i \in \sigma(s^*)} \sb_i \;\; \le \; \prod_{i \in \sigma(s^*)} s^k_i.
\]
So as $\mu^k \downarrow 0$, we find that $(x^*, s^*)$ solves
\[
\max \left\{ \prod_{i \in \sigma(s^*)} \!s_i :
        Ax + s = b, \, s_{\sigma(s^*)} > 0, \, s_{\neg\sigma(s^*)} = 0, \, 
        c^T x = c^T x^* \right\},
\]
which means that it also solves
\[
\max \left\{ \sum_{i \in \sigma(s^*)} \ln(s_i) :
        Ax + s = b, \, s_{\sigma(s^*)} > 0, \, s_{\neg\sigma(s^*)} = 0, \, 
        c^T x = c^T x^* \right\}.
\]
The assumed full column rank of $A$ ensures that $x$ and $s$ are in a 
one-to-one relationship, and specifically, $x = (A^TA)^{-1}A^T(b-s)$
is the same as $Ax + s = b$. Moreover, the constraint 
$s_{\neg\sigma(s^*)} = 0$ means that these variables are fixed, 
and hence, $x$ is in a one-to-one relationship with $s_{\sigma(s^*)}$. 
This means that $s^*_{\sigma(s^*)}$ solves
\begin{eqnarray*}
\lefteqn{
\max \Bigg\{ \sum_{i \in \sigma(s^*)} \ln(s_i) : 
    A(A^TA)^{-1}A^T(b-s) = b-s, } \\
& & \hspace*{40pt} s_{\sigma(s^*)} > 0, \, s_{\neg\sigma(s^*)} = 0, \, 
    c^T (A^TA)^{-1}A^T(b-s) = c^T x^* \Bigg\}.
\end{eqnarray*}
The  Hessian of the objective function is $-S_{\sigma(s^*)}^{-2} \nd 0$.
So the objective function is strictly concave and the solution is unique.
We consequently have that $(x^*, s^*)$ solves~\eqref{eq-primalOptAC}.

Now assume $\mu^k \uparrow \infty$ and again set 
$(x^k, y^k, s^k) = (x(\mu^k), \, y(\mu^k), \, s(\mu^k) )$.
The sequence $(x^k, s^k)$ is in a compact feasible region by assumption,
and it therefore has a cluster point, say $(x^c, s^c)$. We assume without
loss of generality that $(x^k, s^k) \rightarrow (x^c, s^c)$ is a convergent
subsequence as $k \rightarrow \infty$. Let $(\xb, \yb, \sb)$ be feasible 
to the linear programs in~\eqref{eq-lpOpt}. Then $s^k - \sb \in \mbox{col}(A)$ 
and $y^k - \yb \in \mbox{null}(A^T)$, and similar to the orthogonal 
arguments above, we have
\[
\frac{\sb^T \yb}{\mu^k} + m = 
    \sum_{i=1}^m \, \frac{\yb_i}{y^k_i} + 
    \sum_{i=1}^m \, \frac{\sb_i}{s^k_i}.
\]
Note that $s^k_i y^k_i = \mu^k$, so $s^k$ being bounded forces
$y^k_i \rightarrow \infty$ as $k \rightarrow \infty$. We now have that as
$k \rightarrow \infty$,
\[
m = \sum_{i=1}^m \, \frac{\sb_i}{s^c_i},
\]
and the arithmetic-geometric mean inequality gives,
\[
\prod_{i=1}^m \sb_i \le \prod_{i=1}^m s^c_i.
\]
So $(x^c, s^c)$ solves,
\[
\max \left\{ \prod_{i=1}^m s_i \; : \; Ax + s = b, \; s \ge 0 \right\},
\]
and it thus also solves,
\[
\max \left\{ \sum_{i=1}^m \ln(s_i) \; : \; Ax + s = b, \; s > 0 \right\}.
\]
We express this problem uniquely in terms of $s$ using the full column rank 
assumption of $A$, resulting in
\[
\max \left\{ \sum_{i=1}^m \ln(s_i) \; : \; A(A^TA)^{-1}A^T(b-s) = b-s, \; s > 0 \right\}.
\]
The Hessian of the objective function is $-S^{-2} \nd 0$, from 
which we know that $(x^c, s^c)$ is the unique solution 
to~\eqref{eq-primalAc}. So all convergent subsequences of
$(x(\mu^k), \, s(\mu^k) )$ converge to $(x^c, s^c)$, and hence,
$(x(\mu^k), \, s(\mu^k) )$ converges to $(x^c, s^c)$. 
\end{proof}
\noindent Theorem~\ref{thm-converge} guarantees that the closure of the 
central path is
\[
\overline{\bP(G(x), c)} = \bP(G(x), c) \cup \{ x^c, x^*\},
\]
and we assume computationally that paths start at $x^c$, which will be the zero vector by 
design, and ends at $x^*$, i.e that $\mu$ starts at infinity and decreases
to zero.

\subsection{Computing Central Paths} \label{subsec-computingCP}

Our description of the central path gives mathematical assurances, but 
we require a calculation method to generate images. 
The defining system with $G(x) = Ax - b$ is
\[
F(x,s,y,\mu) = \left(
        \begin{array}{c} 
            \displaystyle A^Ty - c \\[6pt]
            \displaystyle Ax + s - b \\[6pt]
            Sy - \mu e
        \end{array}
        \right)
    = \left( \begin{array}{c} 0 \\[6pt] 0 \\[6pt] 0 \end{array} \right),
\;\; \mbox{ with } \;\; (y,s) > 0.
\]
We use Newton's method to estimate $x(\mu)$, $s(\mu)$, and $y(\mu)$ 
from any $x$, $s$, and $y$ by iteratively solving
\begin{equation} \label{eq-stepSys}
\left.
\begin{array}{rcl}
\nabla_{x,s,y} F(x,s,y,\mu) 
    \left( \begin{array}{c} \Delta x \\ \Delta s \\ \Delta y
            \end{array} \right)
& = & \left[ \begin{array}{ccc}
    0 & 0 & A^T \\ A & I & 0 \\ 0 & Y & S \end{array} \right]
    \left( \begin{array}{c} \Delta x \\ \Delta s \\ \Delta y
            \end{array} \right)  \\[26pt]
& = &  \left(
        \begin{array}{c} 
            \displaystyle c - A^T y \\[2pt]
            \displaystyle b - Ax - s \\[2pt]
            \displaystyle \mu e - Sy
        \end{array}
        \right) = -F(x,s,\mu)
\end{array}
\right\}
\end{equation}
and updating $x$, $s$, and $y$ so that
\[
x \leftarrow x + \alpha \, \Delta x, \;\;
s \leftarrow s + \alpha \, \Delta s,
\;\; \mbox{ and } \;\;
y \leftarrow y + \beta \, \Delta y.
\]
The stepsizes $\alpha$ and $\beta$ guarantee nonnegativity by setting 
them to
\begin{eqnarray*}
\lefteqn{
\alpha = \omega \, \min \left\{ 1, \, \min \left\{ x_i / \Delta x_i : 
    \Delta x_i < 0 \right\} \right\} } \\[6pt]
& & \hspace*{40pt} \mbox{ and } \;
\beta = \omega \, \min \left\{ 1, \, \min \left\{ y_i / \Delta y_i : 
    \Delta y_i < 0 \right\} \right\},
\end{eqnarray*}
with $\omega$ being a positive scalar that is less than one.
The value of either inner minimum is $\infty$ if $\Delta x \ge 0$
or $\Delta y \ge 0$, respectively.

Calculating $\Delta x$, $\Delta s$, and $\Delta y$ reduces to solving
an $n \times n$ positive definite system to compute $\Delta x$, from
which we then calculate $\Delta s$ and $\Delta y$.
We first notice the bottom two equations of the linear 
system in~\eqref{eq-stepSys} give
\[
A \Delta x + \Delta s = b - Ax - s
   \Rightarrow A^TS^{-1}Y \left( A \Delta x + \Delta s \right)
                    = A^T S^{-1}Y \left( b - Ax - s \right)
\]
and
\[
Y \Delta s + S \Delta y = \mu e - Sy
   \Rightarrow A^TS^{-1} \left(Y \Delta s + S \Delta y \right)
                    = A^TS^{-1} \left( \mu e - Sy \right). 
\]
Using $A^T \Delta y = c - A^Ty$ from the first equation
in~\eqref{eq-stepSys} and combining these last two implications, we 
find that,
\[
A^T S^{-1}YA \, \Delta x = A^T S^{-1} Y \left( b - Ax - s \right)
    + c - \mu A^T S^{-1} e.
\]
Our geometry permits a feasible starting point $(x,s)$, and
in this case we know that $Ax + s = b$ with $s > 0$. In this situation
we have $A \Delta x + \Delta s = 0$ from the second equation 
in~\eqref{eq-stepSys}, and hence,
$(x + \alpha \Delta x, \, s + \alpha \Delta s)$ remains feasible
after the update because
\[
A (x + \alpha \Delta x) + (s + \alpha \Delta s)
    = Ax + s = b.
\]
Moreover, the equation defining $\Delta x$ becomes
\begin{equation} \label{eq-DeltaxStep}
A^T S^{-1}YA \, \Delta x = c - \mu A^T S^{-1} e.
\end{equation}
We have $A^T S^{-1} YA \pd 0$ because if $v$ is a nonzero vector, then
\[
v^T A^T S^{-1} YA \, v = (\sqrt{S^{-1}Y}A\,v)^T(\sqrt{S^{-1}Y}A\,v)
    = \left\| \sqrt{S^{-1}Y}A \, v \right\|^2 > 0,
\]
with $\sqrt{S^{-1}Y}$ being the diagonal matrix with diagonal elements 
$\sqrt{y_i / s_i}$. The positivity of $\left\| \sqrt{S^{-1}Y}A \, v \right\|$
follows from the fact that $A$ has full column rank, and hence, neither
$Av$ nor $\sqrt{S^{-1}Y}A \, v$ are zero unless $v$ is zero.

Solving~\eqref{eq-DeltaxStep} is efficient and numerically stable,
see~\cite{wright1994,wright1999}, and once we have $\Delta x$, we set
\[
\Delta s = -A \, \Delta x \;\; \mbox{ and } \;\; \mbox{ and } \;\;
\Delta y = \mu S^{-1} e - y - S^{-1} Y \Delta s.
\]
The result is an efficient and numerically stable iterative process
that efficiently converges to $x(\mu)$, $s(\mu)$, and $y(\mu)$. 
We cease iterations once
\begin{equation} \label{eq-tolPath}
\max \left\{ \| F(x,s,\mu) \|, \; \alpha \, \| \Delta x \| \right\}
    < \varepsilon, 
\end{equation}
with $\varepsilon$ being a convergence tolerance. 

Digitally rendering a central path necessitates that we compute
a sequence, say $\mu^k$, along with estimates for $x(\mu^k)$.
We denote the estimates of $x(\mu^k)$, $s(\mu^k)$, and
$y(\mu^k)$ as $x^k$, $s^k$, and $y^k$. The inequality in~\eqref{eq-tolPath} 
ensures that 
\[
\| x^k - x(\mu^k)\| = \alpha \| \Delta x \| \le \varepsilon.
\]
So we guarantee accuracy to the central path at $\mu^k$ by decreasing 
$\varepsilon$, but this assurance leaves open the concern that we could
lose accuracy as a graphics package interpolates estimates to create a
curve. A common measure of how close a point $(x, s, y)$ is
to $(x(\mu^k), s(\mu^k), y(\mu^k))$ is
\begin{equation} \label{eq-standardMetric}
\| Sy - S(\mu^k) \, y(\mu^k) \| = \| Sy - \mu^k e \|,
\end{equation}
see~\cite{vavasis96}, and this measure promotes that we compute 
$\mu^k$ values so that
\[
\max_{\theta} \; \left\{
\left\| \left( (1 - \theta)S^k + \theta S^{k+1} \right)
        \left( (1 - \theta)y^k + \theta y^{k+1} \right)
        - \mu^k e \right\| \, : \; 0 \le \theta \le 1 \right\}
        \le \delta,
\]
with $\delta$ being an allowed deviation from the central path. A straightforward
calculation shows that the maximum is achieved at $\theta = 1/2$, and we initially
devised a recursive bisection technique to ensure that
\begin{equation} \label{eq-midPointbound}
\frac{1}{4} \, 
\left\| \left( S^k + S^{k+1} \right) \left( y^k + y^{k+1} \right)
        - \mu^k e \right\| \le \delta.
\end{equation}
The process starts with $\mu^1$ and $\mu^2$ being respectively large and small
enough so that $x^1 \approx x^c$ and $x^2 \approx x^*$. These approximations
satisfy~\eqref{eq-tolPath}, and if they also satisfy~\eqref{eq-midPointbound},
then we are done. We otherwise compute estimates 
\begin{eqnarray*}
\xh & \approx & x( (\mu^1 + \mu^2) / 2), \\
\sh & \approx & s( (\mu^1 + \mu^2) / 2), \;\; \mbox{ and } \\
\yh & \approx & y( (\mu^1 + \mu^2) / 2),
\end{eqnarray*}
that satisfy~\eqref{eq-tolPath} and renumber so that
\begin{eqnarray*}
(x^3, \, s^3, \, y^3, \mu^3) & \leftmapsto &  (x^2, \, s^2, \, y^2, \mu^2) \;\; \mbox{ and} \\
(x^2, \, s^2, \, y^2, \mu^2) & \leftmapsto &  (\xh, \, \sh, \, \yh, (\mu^1 + \mu^2) /2 ).
\end{eqnarray*}
The process repeats on all $[\mu^k, \mu^{k+1}]$ intervals until each
satisfies~\eqref{eq-midPointbound}.

The $\mu^k$ list from the bisection technique with the standard measure
in~\eqref{eq-standardMetric} proves problematic. The issue is that our
central paths are approximated by the primal estimates, $x^k$, whereas
the dual estimates, $y^k$, dominate the decisions of whether or not to
bisect the intervals. The curvature of the dual central path does not 
align with the curvature of the primal, at least in our geometries, and
the result is a list of $\mu^k$ values that over represent low curvature
portions of the (primal) central path, see Figure~\ref{fig-goodSpacing}.

We correct the overrepresentation of low curvature portions of the central
path by estimating linearity with finite differences. If $x^k$, $x^{k+1}$,
and $x^{k+2}$ are consecutive estimates, then 
\[
T^k = \frac{x^{k+1}-x^k}{\|x^{k+1} - x^k \|}
\;\; \mbox{ and } \;\;
T^{k+1} = \frac{x^{k+2}-x^{k+1}}{\|x^{k+2} - x^{k+1} \|}
\]
approximate the unit tangent vectors of the central path at $\mu^k$ and 
$\mu^{k+1}$. So estimates of the linearity of the central path over 
the intervals $[\mu^k, \, \mu^{k+1}]$ and $[\mu^{k+1}, \, \mu^{k+2}]$ are 
respectively
\[
\kappa^k = \frac{\|T^{k+1} - T^{k}\|}{\|x^{k+1} - x^k\|}
\;\; \mbox{ and } \;\;
\kappa^{k+1} = \frac{\|T^{k+1} - T^{k}\|}{\|x^{k+2} - x^{k+1}\|}.
\]
If $\kappa^k > \delta$, then we assume the central path is not sufficiently
linear over $[\mu^k, \, \mu^{k+1}]$ and we bisect this interval by inserting
the midpoint to our list of $\mu^k$ values. The same occurs if 
$\kappa^{k+1} > \delta$, but in this case we bisect $[\mu^{k+1}, \, \mu^{k+2}]$.

The bisection method based on our curvature estimates reduces the over 
representation of the central path on low curvature segments and increases 
accuracy on high curvature portions. Figure~\ref{fig-goodSpacing} compares
the two techniques. The figure on the left sets $\delta = 0.08$
in~\eqref{eq-midPointbound}, and the red central paths are approximated
with $23$ and $21$ points. Notice how the majority of points accumulate 
near the center of the square and that the linear interpolation is less 
accurate as the path converges toward the upper right. The figure on the
right instead bounds our curvature estimates by $0.5$ and results in $8475$
and $4208$ points respectively for the top and bottom paths. Although the
curvature tactic has more points than the bisection technique, we have the 
favorable outcome that the point estimates accumulate in regions of higher 
curvature.

\begin{figure}
\hfill
\includegraphics[width=0.4\linewidth]{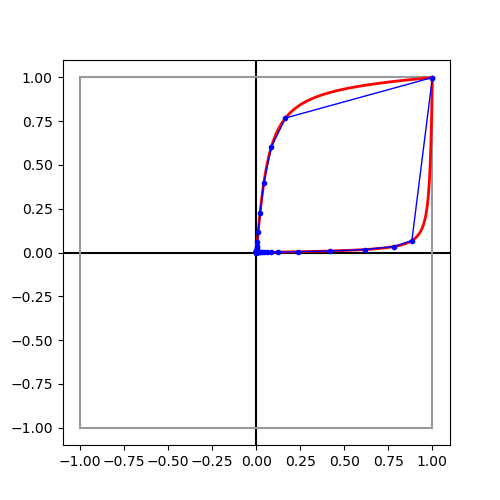}
\hfill
\includegraphics[width=0.4\linewidth]{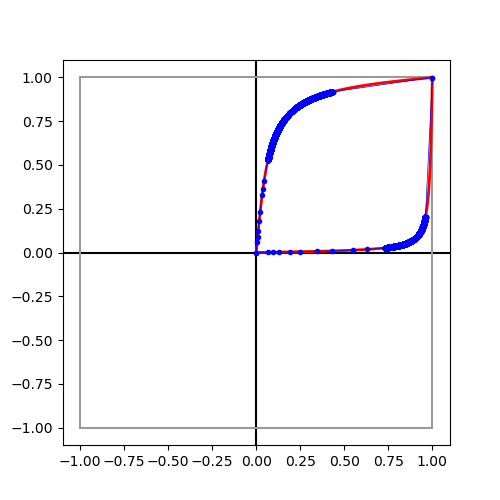}
\hfill
\caption{The figure on the left spaces $\mu$ with~\eqref{eq-midPointbound},
    and the figure on the right spaces $\mu$ by bounding (estimates) of
    curvature.} \label{fig-goodSpacing}
\end{figure}

\subsection{Geometric Considerations} \label{sec-affineTransformation}

Affine transformations aid our geometric effort, and we establish that
the central path of a transformed problem is the same as the transformation of
the central path. This result is not unforeseen or unexpected,
but it should not be assumed without proof because, as previously
noted, the central path depends on the algebraic description of the problem 
and not on the geometry of the problem. We illustrate this fact by observing
that if
\[
A = \left[\begin{array}{rr} 1 & 0 \\ -1 & 0 \\ 0 & 1 \\ 0 & -1 
    \end{array} \right]
\;\;\mbox{ and }\;\;
\hat{A} = \left[\begin{array}{rr} 1 & 0 \\ -1 & 0 \\ 0 & 1 \\ 0 & -1 \\ 1/2 & 1/2
     \end{array} \right],
\]
then $Ax \le e$ and $\hat{A}x \le e$ are individually true if and only if
$-1 \le x_1 \le 1$ and $-1 \le x_2 \le 1$. So
\[
\{ x : Ax \le e \} = \{ x : \hat{A}x \le e \}.
\]
However, setting $G(x) = Ax - b$ and $\hat{G}(x) = \hat{A}x - b$, we find
that $\bP(G(x),c) \neq \bP(\hat{G}(x),c)$, as illustrated in 
Figure~\ref{fig-neqCPs} with $c^T = (1, 2)$. The basic idea is that
geometric intuition should draw caution because geometric perspective 
is coupled with, but disjoint from, algebraic exposition.

\begin{figure}[ht]
\begin{center}
\includegraphics[width=0.3\linewidth]{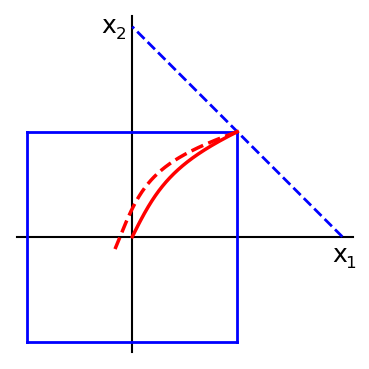}
\caption{The solid central path ignores the dashed redundant 
constraint, whereas the dashed central path includes the redundant
constraint.} \label{fig-neqCPs}
\end{center}
\end{figure}

Let $T(x) = Bx + d$ be an affine transformation with $B$ being an
invertible $n \times n$ matrix. Applying $T(x)$ to an element of
the polytope $\{x : Ax \le b\}$, we find that $T(x) = Bx + d = z$ if
and only if $x = B^{-1}(z-d)$. This relationship gives the algebraic
description of the image polytope as $\{z : AB^{-1} z \le b + AB^{-1}d\}$.
We now prove that central paths in $\{ x : Ax \le b\}$ correspond with
central paths in the image polytope.

\begin{theorem} \label{thm-affineScaling}
Let $B$ be an invertible $n \times n$ matrix, and let $T(x) = Bx + d$.
Set $G(x) = Ax - b$ and $\hat{G}(z) = AB^{-1} z - b - AB^{-1}d$. Then
\[
T\left(\bP(G(x), \, c)\right) = \bP\left(\hat{G}(z), \, (B^{-1})^Tc\right).
\]
\end{theorem}
\begin{proof}
We first note that $x(\mu) \in \bP(G(x), \, c)$ if and only if
there are unique $s(\mu)$ and $y(\mu)$ that satisfy
\begin{eqnarray*}
Ax + s & = & b, \; s \ge 0 \\
A^T y & = & c, \; y \ge 0 \\
Ys & = & \mu e.
\end{eqnarray*}
Set $T(x(\mu)) = z(\mu)$ so that $x(\mu) = B^{-1}(z(\mu) - d)$. Then
$z(\mu)$ satisfies
\begin{eqnarray*}
AB^{-1}z + s & = & b + AB^{-1}d, \; s \ge 0 \\
(B^{-1})^T A^T y & = & (B^{-1})^T c, \; y \ge 0 \\
Ys & = & \mu e,
\end{eqnarray*}
which are the necessary and sufficient conditions for
$z(\mu)$ to be the $\mu$ element of 
$\bP\left(\hat{G}(z), \, (B^{-1})^Tc\right)$.
\end{proof}
\noindent We comment that the proof is stronger then the theorem
statement because we have actually shown that the $\mu$ element of
$\bP(G(x), \, c)$ maps to the $\mu$ element of
$\bP\left(\hat{G}(z), \, (B^{-1})^Tc\right)$. 

Theorem~\ref{thm-affineScaling} illustrates a mathematical and computational
theme of our work, which is that we generate images with standard
geometries and then translate them to create aesthetic images. Our most
common two-dimensional polytope is the regular $k$-gon defined by $G^k(x) = A^k x - e$,
with the $i$-th row of the $k \times 2$ matrix $A$ being
\[
A^k_i = (\cos((i-1)2\pi/k), \; \sin((i-1)2\pi/k).
\]
We use the superscript $k$ on both $G$ and $A$ to indicate this polygon.
Consider the rotation matrix,
\[
R(\theta) = \left[ \begin{array}{rr}
    \cos(\theta) & \sin(\theta) \\[2pt] -\sin(\theta) & \cos(\theta)
    \end{array} \right],
\]
and set $c^i(\theta) = A_i R(\theta)$, with $A_i$ being the $i$-th
row of $A$. Figure~\ref{fig-triangle} displays $\bP(G^3(x), c^i(\theta))$
for $i \in \{1, 2, 3\}$ and $\theta \in \{ 0.009, -0.009, 0.18, -0.18\}$,
and Figure~\ref{fig-hexagon} displays $\bP(G^6(x), c^i(\theta))$
for $i \in \{1, 2, \ldots, 6 \}$ and the same values of $\theta$.
Figure~\ref{fig-star} displays a star created by translating
the $3$-gon paths around the $6$-gon with the affine translations,
\[
R\left( \frac{i\,\pi}{3} \right) \left(
    \left[ \begin{array}{cc} 1/3 & 0 \\ 0 & 1/3 \end{array} \right]
    + \left( \begin{array}{r} -2 \\ 0 \end{array} \right) \right),
\]
with $i \in \{0, 1, 2, \ldots, 5\}$. Theorem~\ref{thm-affineScaling} ensures 
that the paths under translation  are indeed paths within the 
translated polytopes. We comment that the objects in 
Figure~\ref{fig-translation} resemble $k$-gons by design, but they 
are not $k$-gons per se and are instead paths within $k$-gons.

\begin{figure}
\begin{center}
\begin{subfigure}[t]{0.26\linewidth}
\begin{center}
\includegraphics[width=\linewidth]{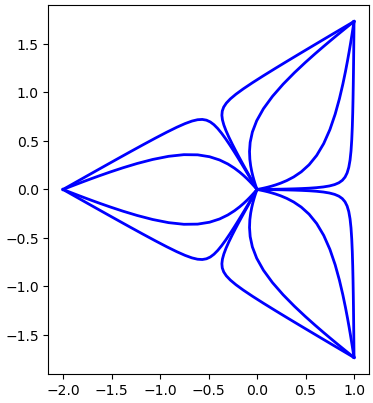}
\caption{Paths in a $3$-gon.} \label{fig-triangle}
\end{center}
\end{subfigure}
\hspace*{10pt}
\begin{subfigure}[t]{0.26\linewidth}
\begin{center}
\includegraphics[width=\linewidth]{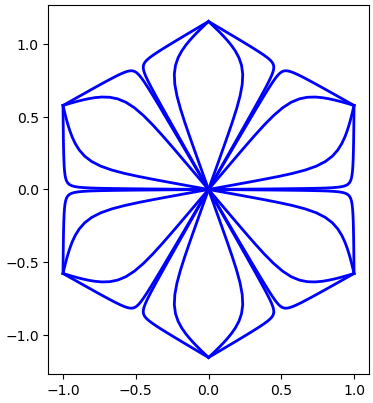}
\caption{Paths in a $6$-gon.} \label{fig-hexagon}
\end{center}
\end{subfigure}
\hspace*{10pt}
\begin{subfigure}[t]{0.332\linewidth}
\begin{center}
\includegraphics[width=\linewidth]{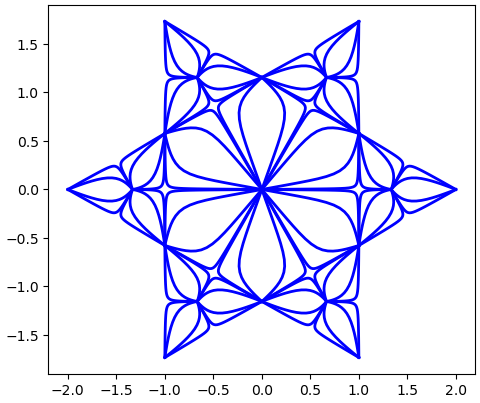}
\caption{A star figure.} \label{fig-star}
\end{center}
\end{subfigure}
\caption{Translations of $3$-gon paths form a star around a $6$-gon.}
\label{fig-translation}
\end{center}
\end{figure}

Another use of Theorem~\ref{thm-affineScaling} is the ability to create
a font from central paths. We create characters from collections of
`strokes,' each of which equates with a translation of
$\bP(G^4(x),c)$, with $c$ being decided to give an appropriate
shape. The authors have taken to the
moniker ``InteriArt," and Figure~\ref{fig-InteriArt} is
an example of this term in characters rendered from $4$-gon
central paths. We also point out that we have replaced the
traditional black square at the end of a proof with an emblem
of paths in a $4$-gon as a bit of artistic levity.

\begin{figure}[ht]
\begin{center}
\includegraphics[width=0.6\linewidth]{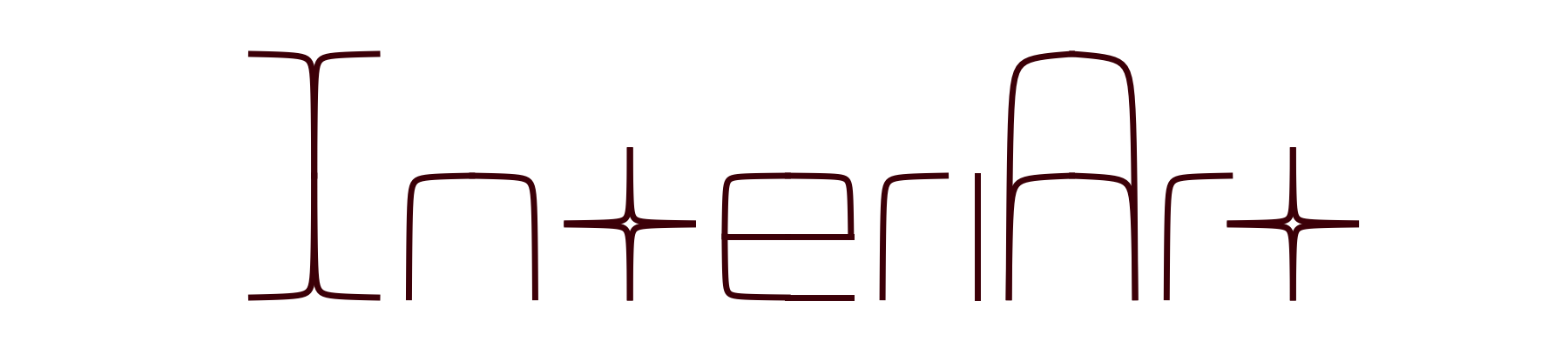}
\caption{A font created from central paths.}
    \label{fig-InteriArt}
\end{center}
\end{figure}

We have already seen in Figure~\ref{fig-neqCPs} 
that different descriptions of the same geometry can lead to different 
central paths, and hence, it would be somewhat natural and analogous to 
assume that different geometries, which would necessarily have different 
algebraic descriptions, would also have different central paths.
However, this is not a universal truth, and our final result 
shows that it is possible for central paths within different
geometries and different algebraic descriptions to equate with
each other. This result permits us to generate two-dimensional surfaces 
in three dimensions by generating a single path in a cube and then continually 
rotating it in three dimensions to form a surface.

Let
\[
\begin{array}{l}
\renewcommand{\arraystretch}{1.1}
G^1(x) = \left(\begin{array}{r}
    x_1 - b_1 \\ x_2 - b_2 \\ x_3 - b_3 \\
    -x_1 - b_1 \\ -x_2 - b_2 \\ -x_3 - b_3 \end{array} \right)
       \; \mbox{ and } \;
G^2(x) = \left(\begin{array}{c}
    \;\;\,x_1 - b_1 \\ -x_1 - b_1 \\  x_2^2 + x_3^2 - b_2^2 
    \end{array} \right),
\end{array}
\]
and consider the geometry in Figure~\ref{fig-cylinder}, which assumes 
$b_1 = b_2 = 1$ for motivational purposes. The three-dimensional 
cube, $\{ x : -1 \le x_i \le 1, i = 1, 2, 3\}$, is $\{x : G^1(x) \le 0\}$, 
and this cube circumscribes the three-dimensional cylinder, 
$\{ x : G^2(x) \le 0\}$. Theorem~\ref{thm-cylinder} shows that central 
paths of the rotated cube equate with central paths in the cylinder.

\begin{figure}[th]
\begin{center}
\includegraphics[width=0.45\linewidth]{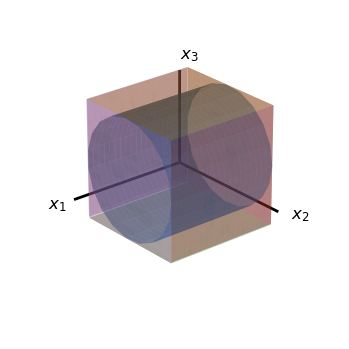}
\vspace*{-30pt}
\caption{A right circular cylinder circumscribed by a cube.}
\label{fig-cylinder}
\end{center}
\end{figure}

\begin{theorem} \label{thm-cylinder}
Assume $c = (c_1, \, c_2, \, c_3)^T$ is such that $c_1 \neq 0$
and that at least one of $c_2$ or $c_3$ is nonzero, and let $R$ be 
the rotation matrix in the $x_2$ and $x_3$ plane,
\[
R = \frac{1}{\sqrt{c_2^2 + c_3^2}}
    \left[ \begin{array}{ccc}
    \sqrt{c_2^2+c_3^2} & 0 & 0 \\
    0 & \;\;\,c_2 & c_3 \\
    0 & -c_3 & c_2
    \end{array} \right].
\]
Let $T(x) = Rx$ so that $T^{-1}(x) = R^{-1}x$. Then
\begin{eqnarray*}
T^{-1} \left( \bP\left( G^1(x), \, T(c) \right) \right) =
       \bP \left( G^2(x), c \right).
\end{eqnarray*}
\end{theorem}
\begin{proof}
The necessary and sufficient conditions in~\eqref{eq-necSuffCond}
for $\bP(G^1, \, T(c) )$, in block matrix form, are
\[
\left[ \begin{array}{cc} I & -I \end{array} \right]
    \left( \begin{array}{c} y_1 \\ y_2 \end{array} \right) = T(c), \;\;\;\;
\left[ \begin{array}{r} I \\ -I \end{array} \right] x +
    \left( \begin{array}{c} s_1 \\ s_2 \end{array} \right) =
    \left( \begin{array}{c} b \\ b \end{array} \right), \;\mbox{ and }\;
Sy = \mu e.
\]
The middle equation asserts that $s_1 = b - x$ and $s_2 = b + x$, and
substituting these into the third equality gives
\[
y_1 = \mu \left( B - X \right)^{-1} e \;\; \mbox{ and } \;\;
    y_2 = \mu \left( B + X \right)^{-1} e,
\]
with both $B+X$ and $B-X$ being invertible because $0 < \mu < \infty$
ensures that $x(\mu)$ is strictly feasible, i.e. $-b < x < b$. The first
equality now gives that 
\[
\mu \left( (B-X)^{-1} - (B+X)^{-1} \right) e = T(c)
\]
defines $\bP(G^1(x), T(c))$. Recognizing that
\begin{eqnarray*}
\lefteqn{(B-X)^{-1} - (B+X)^{-1}} \\ 
   & = & \left((B-X)(B+X)\right)^{-1}(B-X)(B+X) 
            \left((B-X)^{-1} - (B+X)^{-1}\right) \\
   & = & \left(B^2-X^2\right)^{-1}
            \left((B+X) - (B-X)\right) \\
   & = & 2 \left( B^2 - X^2 \right)^{-1} X,
\end{eqnarray*}
we find that the elements of $\bP(G^1(x), T(c))$ satisfy
\[
2\mu \left(B^2 - X^2 \right)^{-1} x = T(c),
\]
which we re-express as
\begin{equation} \label{eq-quadForm}
X^2 T(c) + 2\mu x - B^2 T(c) = 0.
\end{equation}
We now have that the unique solution is
\[
x_i(\mu) = \left\{ \begin{array}{cl}
    \displaystyle
    -\frac{\mu}{T(c)_i} + \frac{1}{|T(c)_i|}\sqrt{ \mu^2 + T(c)_i^2 \, b_i^2 }, & T(c)_i \neq 0 \\[16pt]
    \displaystyle
    0, & T(c)_i = 0,
    \end{array} \right.
\]
with $T(c)_i$ being the $i$-th component of $T(c)$. From the fact that
\[
T(c) = \left( c_1 \, , \;\; \sqrt{c_2^2 + c_3^2} \,, \;\; 0 \right)^T\!\!,
\]
we have that the elements of $T^{-1}\left(\bP(G^1(x), T(c))\right)$ are
\begin{equation} \label{eq-rotatedCubeSolution}
R^{-1}x(\mu) =  \left( \begin{array}{c}
        \displaystyle
        -\frac{\mu}{c_1}
         + \frac{1}{|c_1|} \sqrt{\mu^2 + c_1^2 b_1^2 }\\[24pt]
        \displaystyle
        c_2 \left( \frac{-\mu + \sqrt{\mu^2 + (c_2^2 + c_3^2)b_2}}{c_2^2+c_3^2} \right) \\[24pt]
        \displaystyle
        c_3 \left( \frac{-\mu + \sqrt{\mu^2 + (c_2^2 + c_3^2)b_2}}{c_2^2+c_3^2} \right)
        \end{array} \right).
\end{equation}

The necessary and sufficient conditions for $\bP \left( G^2(x), c \right)$ are
\begin{eqnarray*}
\lefteqn{
\left[ \begin{array}{rrr} 1 & -1 & 0 \\
    0 & 0 & 2x_2 \\ 0 & 0 & 2x_3 \end{array} \right]
    \left( \begin{array}{c} y_1 \\ y_2 \\ y_3 \end{array} \right)
    = \left( \begin{array}{c} c_1 \\ c_2 \\ c_3 \end{array} \right),} \\[10pt]
& & \hspace*{40pt}
    \left( \begin{array}{c} \;\;\,x_1 \\[2pt] -x_1 \\[2pt]
        x_1^2 + x_2^2 \end{array} \right) + 
        \left( \begin{array}{c} s_1 \\[2pt] s_2 \\[2pt] s_3 \end{array} \right) =
        \left( \begin{array}{c} b_1 \\[2pt] b_1 \\[2pt] b_2^2 \end{array} \right), 
        \; \mbox{ and } \;
Ys = \mu e.
\end{eqnarray*}
Solving the second equality for $s$ and substituting this expression into the third equality gives
\[
y = \mu
\left[ \begin{array}{ccc}
    b_1-x_1 & 0 & 0 \\ 0 & b_1+x_1 & 0 \\ 0 & 0 & b_2^2-x_2^2-x_3^2
    \end{array} \right]^{-1} 
    \left( \begin{array}{c} 1 \\ 1 \\ 1 \end{array} \right),
\]
with the inverse again being guaranteed because $0 < \mu < \infty$
ensures strict feasibility. Substituting this expression for $y$
into the first condition gives
\[
\mu \left[ \begin{array}{rrr} 1 & -1 & 0 \\[4pt]
    0 & 0 & 2x_2 \\[4pt] 0 & 0 & 2x_3 \end{array} \right] 
    \left( \begin{array}{c}
        1\, / \,(b_1-x_1) \\[4pt] 1\, / \, (b_1+x_1) 
            \\[4pt] 1\, / \,(b_2^2-x_2^2-x_3^2)
        \end{array} \right)
    = \left( \begin{array}{c} c_1 \\[4pt] c_2 \\[4pt] c_3 \end{array} \right).
\]
We rewrite these equations as
\begin{eqnarray*}
c_1 x_1^2 + 2\mu x_1 - b_1^2 c_1 & = & 0, \\
c_2(b_2 -x_2^2-x_3^2) & = & 2x_2 \mu, \mbox{ and} \\
c_3(b_3 -x_2^2-x_3^2) & = & 2x_3 \mu.
\end{eqnarray*}
The first equation identifies $x_1$ exactly like~\eqref{eq-quadForm}, 
and hence, $x_1(\mu)$ is as stated in~\eqref{eq-rotatedCubeSolution}. 
The last two equations are coupled quadratics, and direct substitution
of the second and third components of~\eqref{eq-rotatedCubeSolution}
satisfy these equations.
\end{proof}

\section{Artistic Outcomes} \label{sec-art}

We now showcase several of our artistic projects, but we first note that 
the labor of creating quality art has been equal to, if not in excess of, 
our mathematical and computational efforts. Every art project has brought 
challenges and failures, and we have learned from each 
experience. Sometimes the hassle has been confronting technological limitations, sometimes 
it has been correcting standard image processing schemes, and sometimes it has been
the slow, deliberate, and pedantic industry required to do most anything well. The projects 
that follow illustrate items that are the result of numerous previous attempts that were made
before `getting things right.' We catalog the effort associated with our largest piece 
to date in the attached supplement. Creating quality art is work, but it is also fulfilling 
and, in our case, it promotes how the mathematics of optimization is not only 
useful, but beautiful.

\subsection{Two-Dimensional Art and Tilings} \label{sec-2Dart}

Central paths in two-dimensional $k$-gons combine to form
appealing patterns useful in a variety of art projects, and we use
these patterns to create a number of items. Most especially,
two-dimensional paths lend themselves to laser cutting and etching
and to 3D-printing. Our first 3D-print had random paths in a $4$-gon, with
each path being $\bP(G^4(x), c)$ for a random $c$, see Figure~\ref{fig-orig3Dprint}. 
Distributional qualities adjust appearance, and in this case we used two 
different distributions. If we visually divide the image into four ``leaves" for
descriptive purposes, then each leaf corresponds with one of the four vertices.
Two paths define and bound each leaf, and we use a normal distribution to decide
these paths. Let $\alpha \sim N(\eta, \sigma^2)$, and notice that each leaf is bracketed 
by paths defined by two rows of $A^k$. This is because $\bP(G^4(x), A^4_i)$ 
is a line from the origin to the midpoint of the facet defined by the supporting 
hyperplane $\{ x : A^4_i x = 1\}$. For instance, paths moving to 
the upper left have $c$ vectors that are convex combinations of 
$A^4_2 = (0,1)$ and $A^4_3 = (-1,0)$. The paths in this object have $\eta = 0.0425$ 
and $\sigma = \eta/3$, and a pair of random samples, say $\alpha_1$ and $\alpha_2$, define 
the two outer paths of each leaf. So the paths that define the leaf to the upper left
have the form
\[
c_1^T =  (1 - \alpha) A^4_2 + \alpha A^4_3 
    \;\; \mbox{ and } \;\; c^T_2 = (1 - \alpha) A^4_3 + \alpha A^4_2.
\]
There is about a $0.15\%$ chance that one of these paths is part of
a neighboring leaf, which seems to the authors akin to generating a four leaf 
clover. The chance of such an event is obviously dependent on the distribution and 
its parameters, and toying with alternatives leads to different qualities.

Paths within a leaf are generated with a uniform distribution,
and in this example we use $\beta \sim U(2\eta, 1-2\eta)$, with $\eta = 0.0425$ 
as before. Each leaf has three or four paths internal to the 
leaf, and each of these is defined by a $c$ vector of the form
\[
c^T = (1 - \beta) A^4_i + \beta A^4_j.
\]
This stochastic process of path generation allows us to generate numerous patterns
in a $k$-gon, and we use this technique in some of our projects.

Many of our patterns tile paths in adjoining polygons, but the resulting tilings
are not polygonal per se. They are instead tilings of paths in polygons. This leaves
patterns that only touch, or intersect, at terminal points of a path.
Figure~\ref{fig-trivet} is a version of the Pythagorean tiling and illustrates
this phenomena. The Pythagorean tiling cascades two $4$-gons, but the piece in
Figure~\ref{fig-trivet} is created by paths in the $4$-gons, making the borders
of the $4$-gons intuitive even though they are not there.

\begin{figure}[t]
\begin{center}
\begin{subfigure}[t]{0.45\linewidth}
\begin{center}
\includegraphics[width=0.6\linewidth]{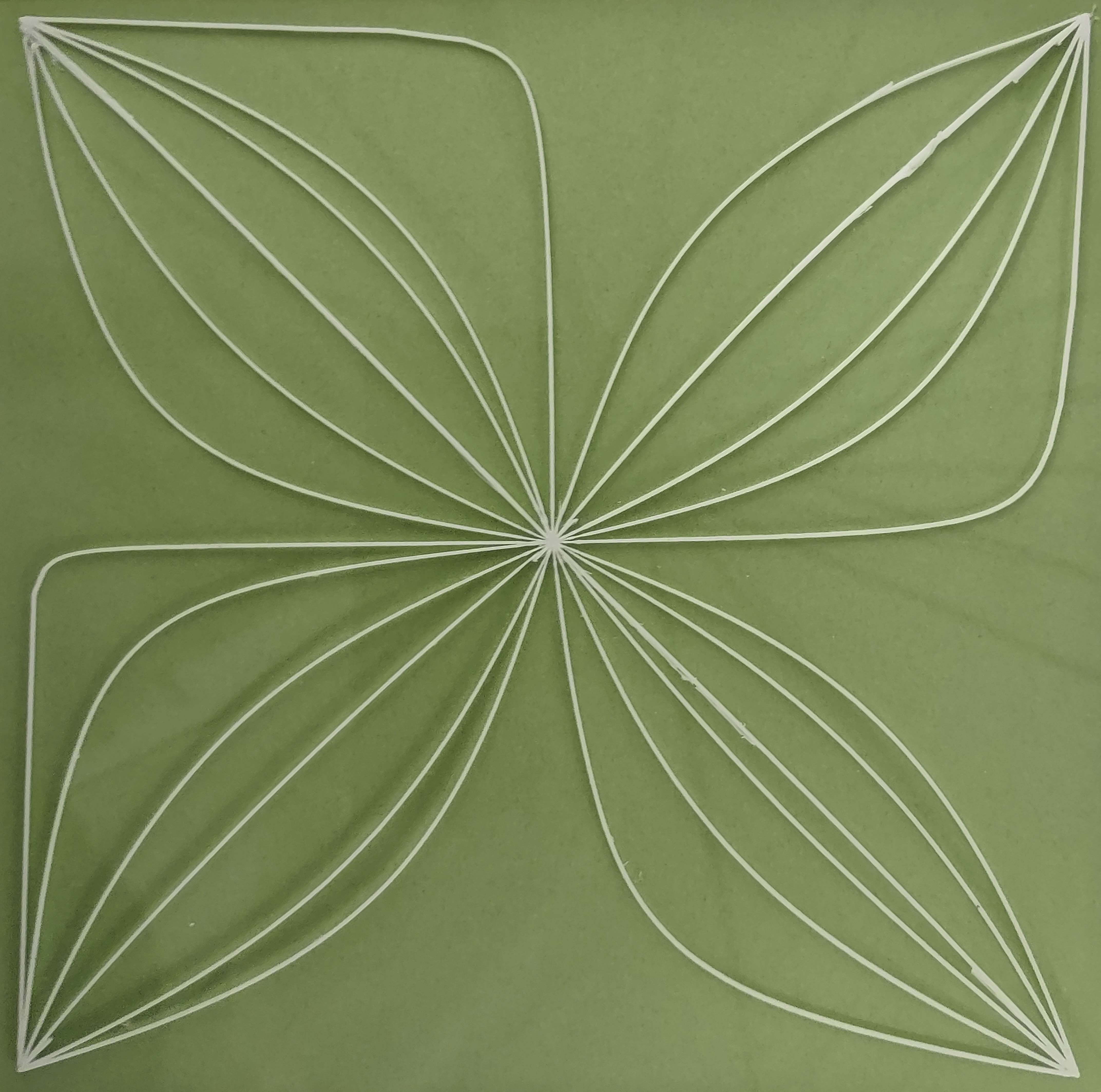}
\caption{\footnotesize A 3D-print of paths in a $4$-gon.} \label{fig-orig3Dprint}
\end{center}
\end{subfigure}
\hspace*{0.05\linewidth}
\begin{subfigure}[t]{0.45\linewidth}
\begin{center}
\includegraphics[width=0.6\linewidth]{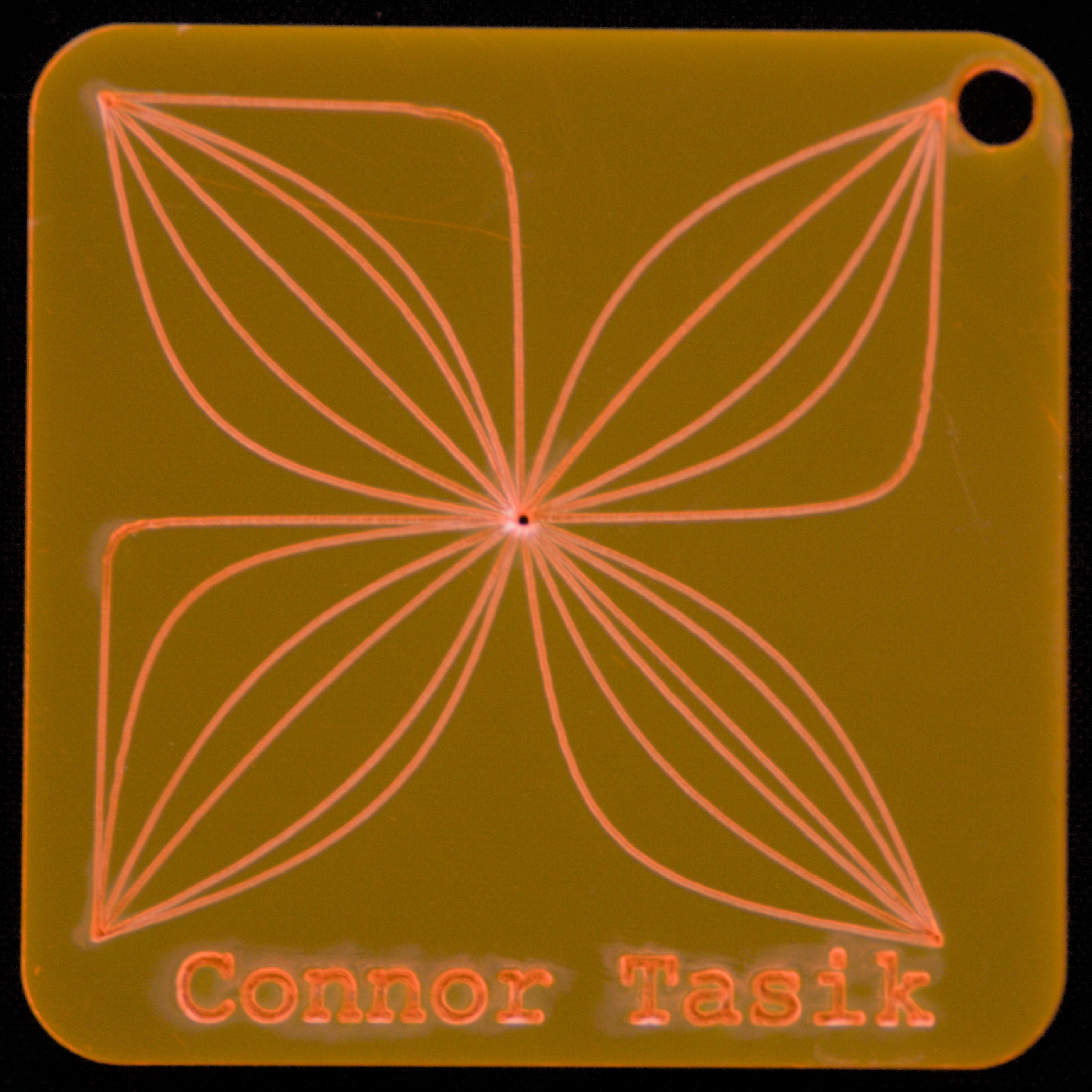}
\caption{\footnotesize A backpack tag.} \label{fig-backpackTag}
\end{center}
\end{subfigure}
\end{center}
\caption{Our original 3D-print of two-dimensional paths and a
    backpack tag from the same image.} \label{fig-first2Dart}
\end{figure}

\subsection*{Etchings and Laser Cutouts}

The technique of generating random paths in a $4$-gon extends directly to $k$-gons, and we have
used this concept to promote math, art, and engineering to middle school students during
Rose-Hulman's 2024 Sonia Math Day. We added an electrical engineering aspect to the project 
by designing a circuit that allowed participants to adjust parameters like $k$ and the number 
of paths in a leaf, and they generated patterns until one struck their fancy. We 
etched these patterns onto colored acrylic medallions with a laser cutter to create 
backpack tags, resulting in mementos of a chic application of math and 
technology. Figure~\ref{fig-backpackTag} is an example backpack tag of the paths 
in Figure~\ref{fig-orig3Dprint}. We have also used patterns of central paths to etch 
leather to create coasters and key-chain fobs, and we have laser-cut wood trivets, see 
Figure~\ref{fig-laserCutter}. We continue to explore other possibilities, and in particular, 
we hope to tile glass etchings to create a work of stained glass. An acrylic
prototype is in Figure~\ref{fig-stainedGlass}. 

\begin{figure}[ht]
\begin{center}
\begin{subfigure}[t]{0.315\linewidth}
\begin{center}
\includegraphics[width=0.61\linewidth]{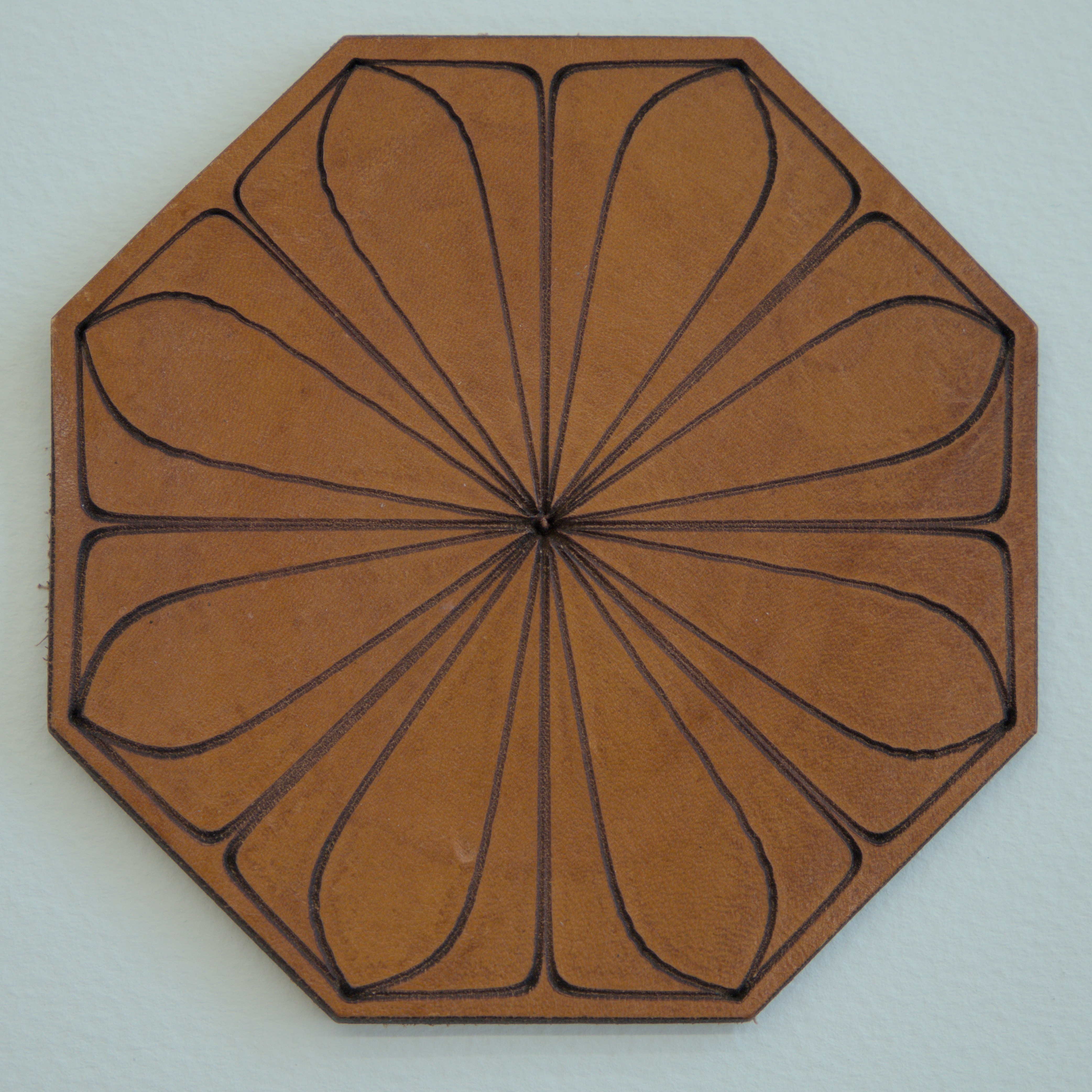}
\caption{\footnotesize An etched leather coaster.} \label{fig-coaster}
\end{center}
\end{subfigure}
\hspace*{0.01\linewidth}
\begin{subfigure}[t]{0.315\linewidth}
\begin{center}
\includegraphics[width=0.478\linewidth]{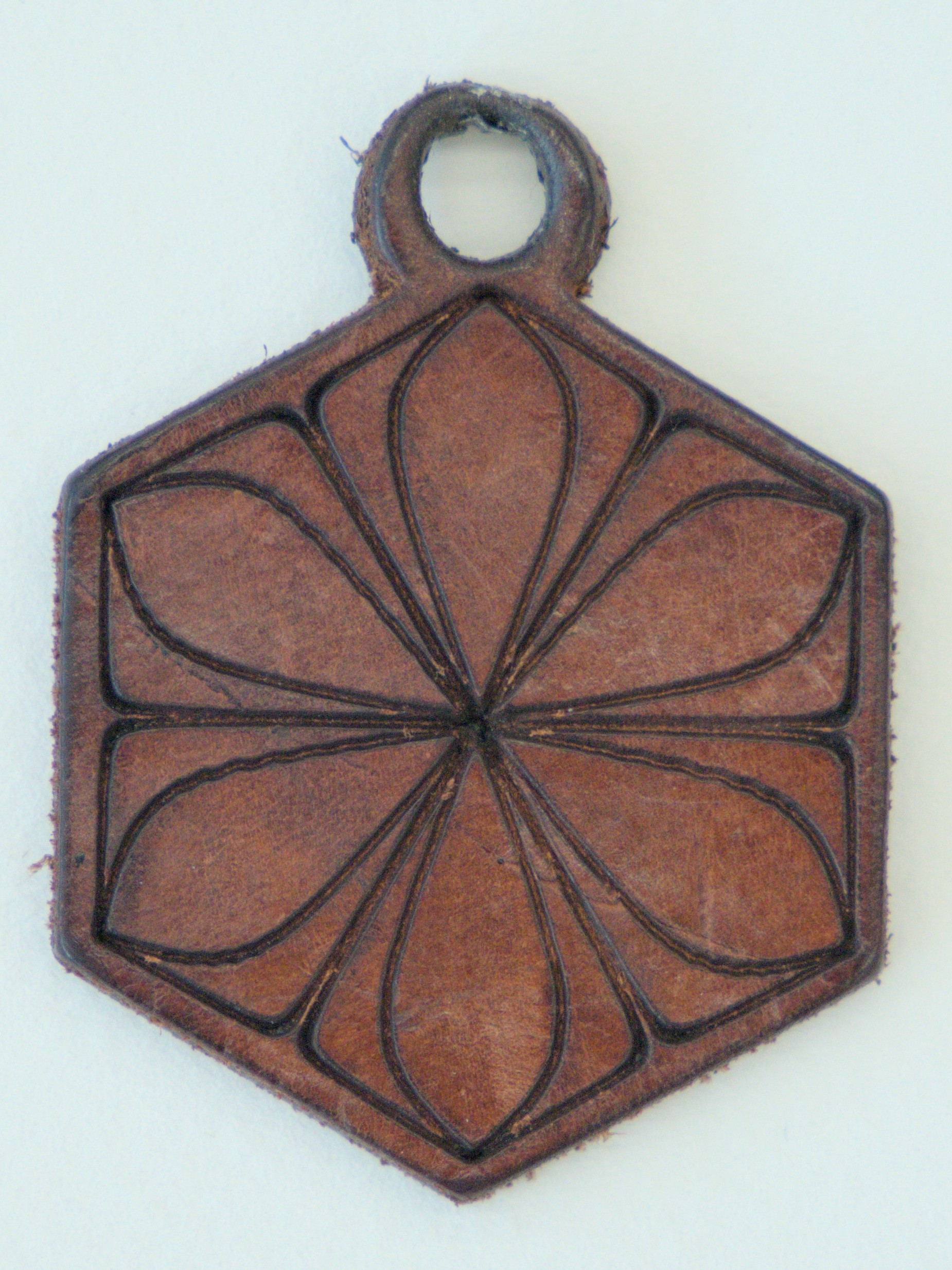}
\caption{\footnotesize An etched key-chain fob.} \label{fig-keychain}
\end{center}
\end{subfigure}
\hspace*{0.01\linewidth}
\begin{subfigure}[t]{0.315\linewidth}
\begin{center}
\includegraphics[width=0.95\linewidth]{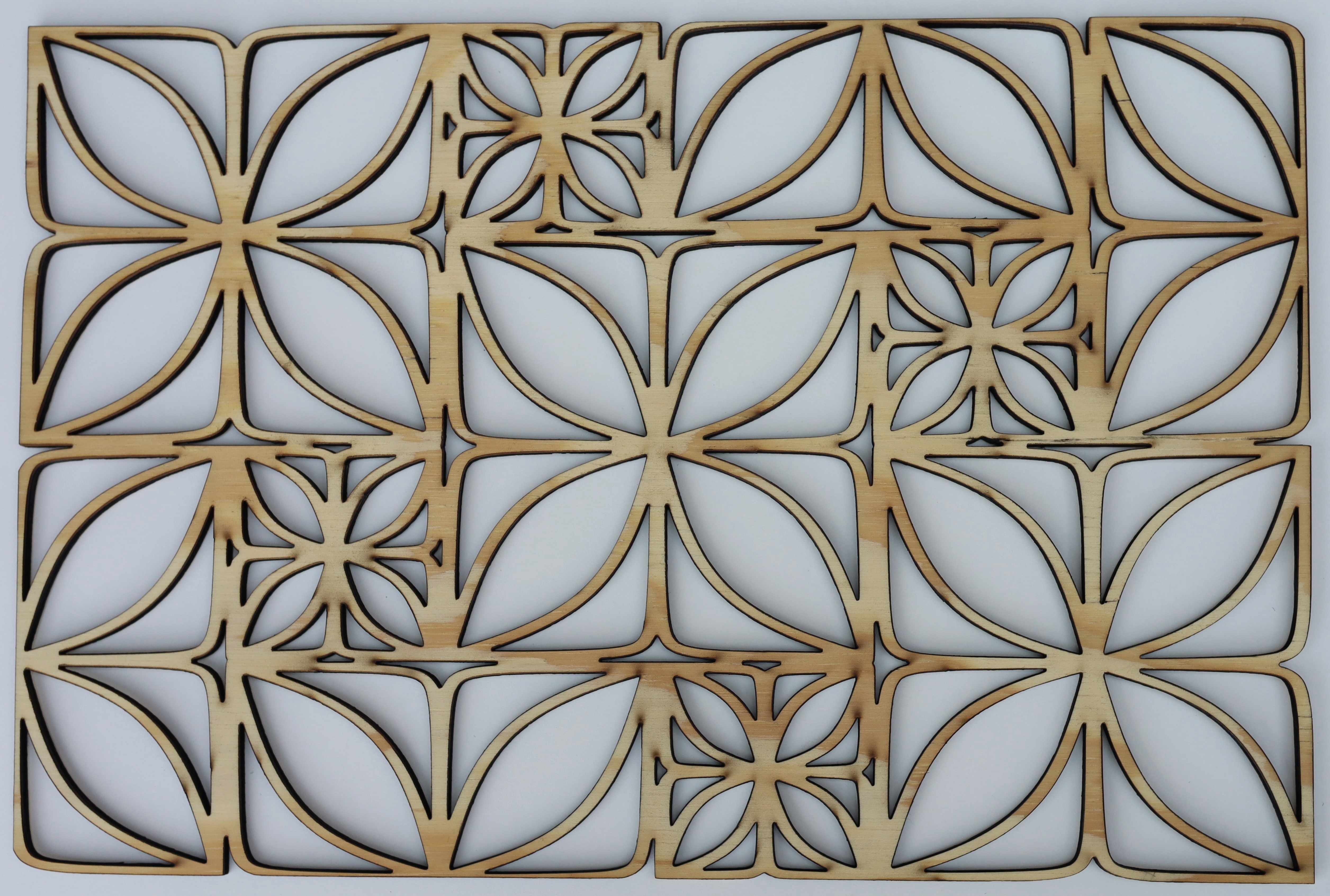}
\caption{\footnotesize A laser-cut wood trivet.} \label{fig-trivet}
\end{center}
\end{subfigure}
\end{center}
\vspace*{-10pt}
\caption{Art items created with a laser cutter/etcher.} 
    \label{fig-laserCutter}
\end{figure}

\begin{figure}[ht]
\begin{center}
\includegraphics[width=0.37\linewidth]{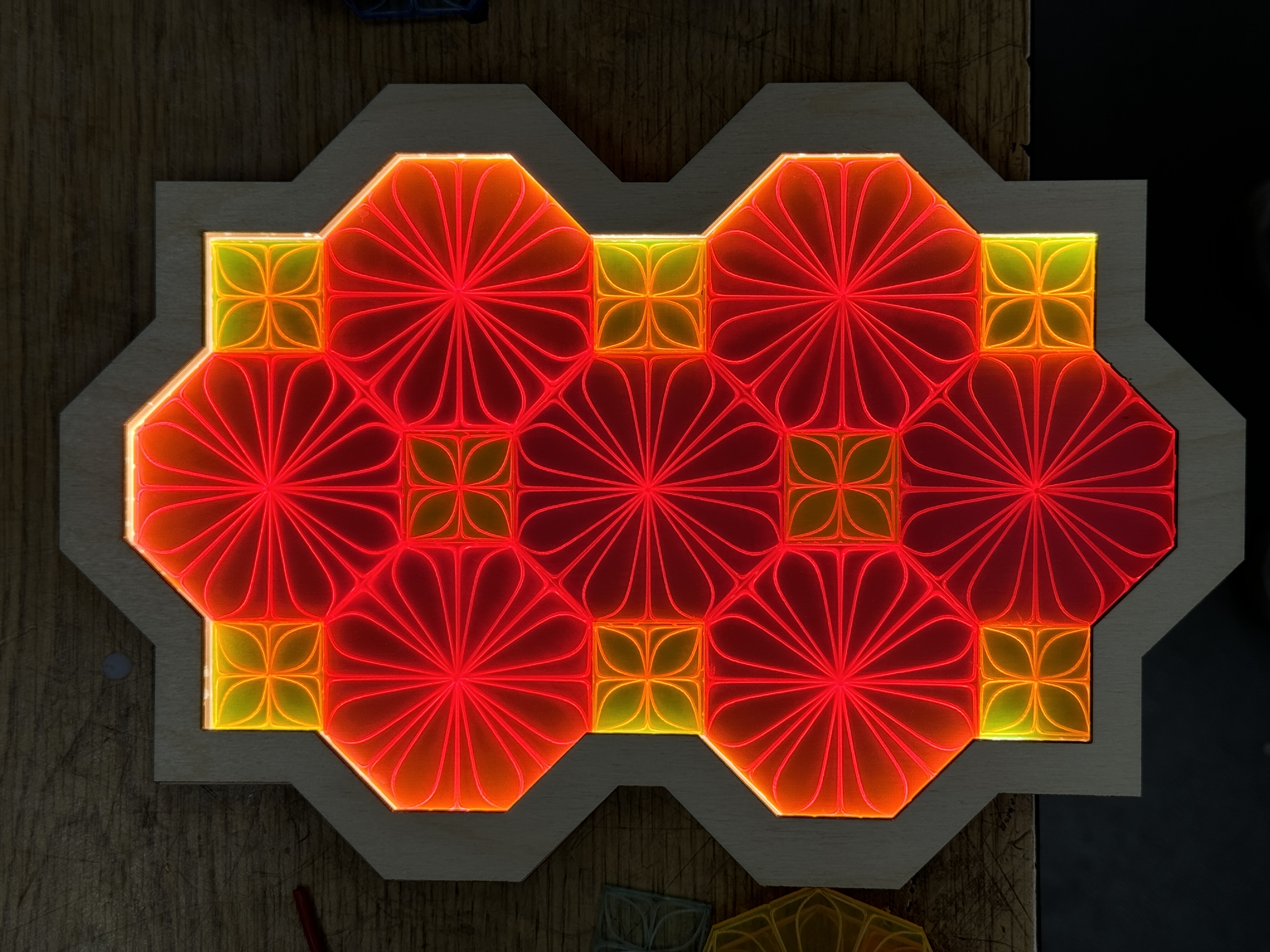}
\end{center}
\vspace*{-10pt}
\caption{An acrylic prototype of a stained glass piece.} \label{fig-stainedGlass}
\end{figure}

\subsection*{3D-Printing}

3D-printers use several modalities to construct objects, and for the moment, 
we only consider the most common printing method of Fused Deposition Modeling (FDM).
FDM printers are not capable of creating three-dimensional structures with paths in, 
for instance, a Platonic solid because they would require an impractical number of 
structural supports. However, modern FDM printers do provide a cost-effective way to 
create three-dimensional objects associated with two-dimensional patterns, which for 
us are tilings of central paths from $k$-gons. The printed items combine pattern, color,
and shape, and we have used them as holiday decorations and decorative trinkets. We have 
also printed and assembled random patterns to create a bouquet of daisy- and thistle-like 
objects, a project that we discuss below. Illustrative objects are in 
Figure~\ref{fig-fdmPrints}.

\begin{figure}[ht]
\begin{center}
\begin{subfigure}[t]{0.315\linewidth}
\begin{center}
\includegraphics[width=0.71\linewidth]{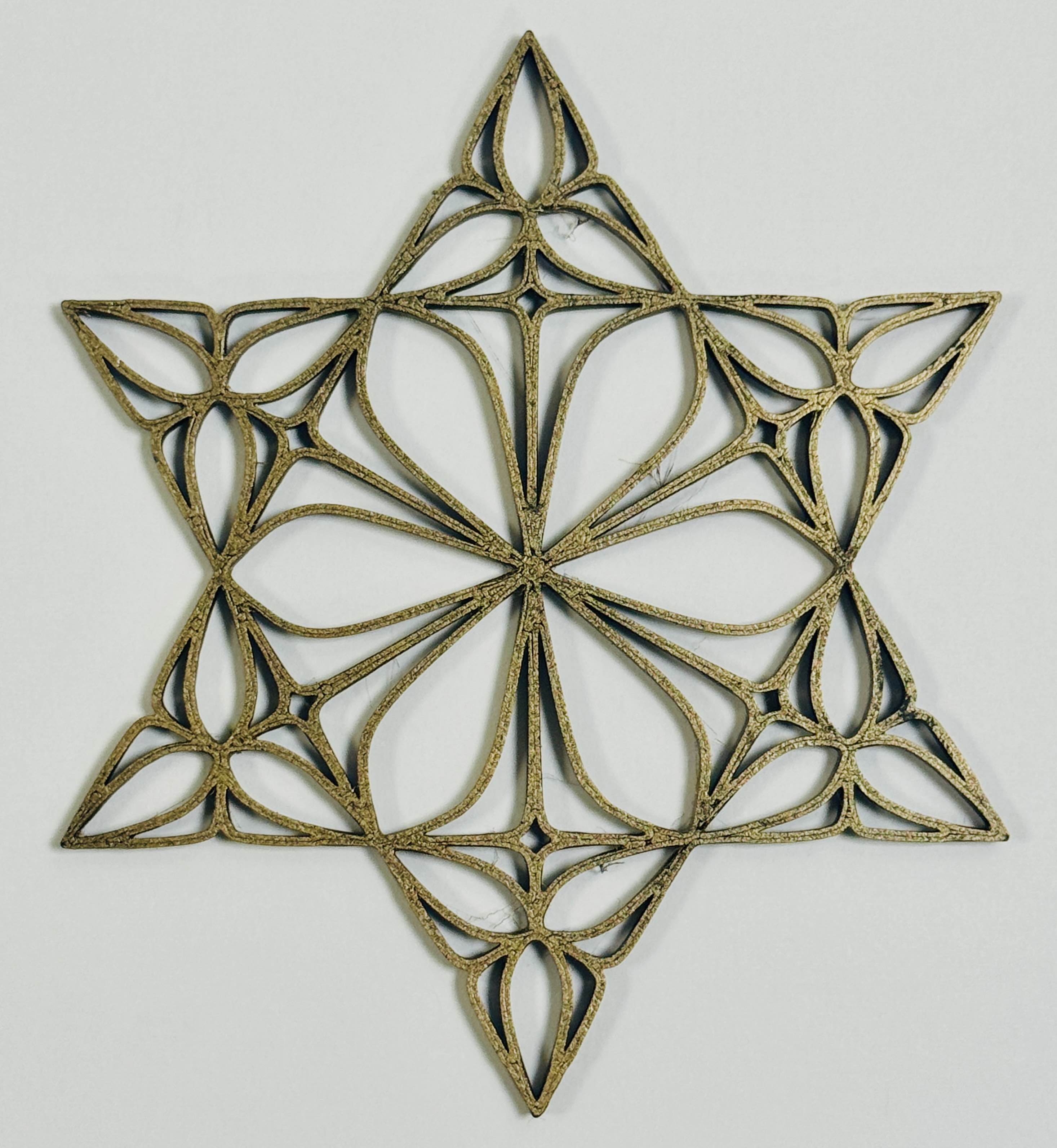}
\caption{\footnotesize A holiday decoration.} \label{fig-holiday}
\end{center}
\end{subfigure}
\hspace*{0.01\linewidth}
\begin{subfigure}[t]{0.315\linewidth}
\begin{center}
\includegraphics[width=0.7\linewidth]{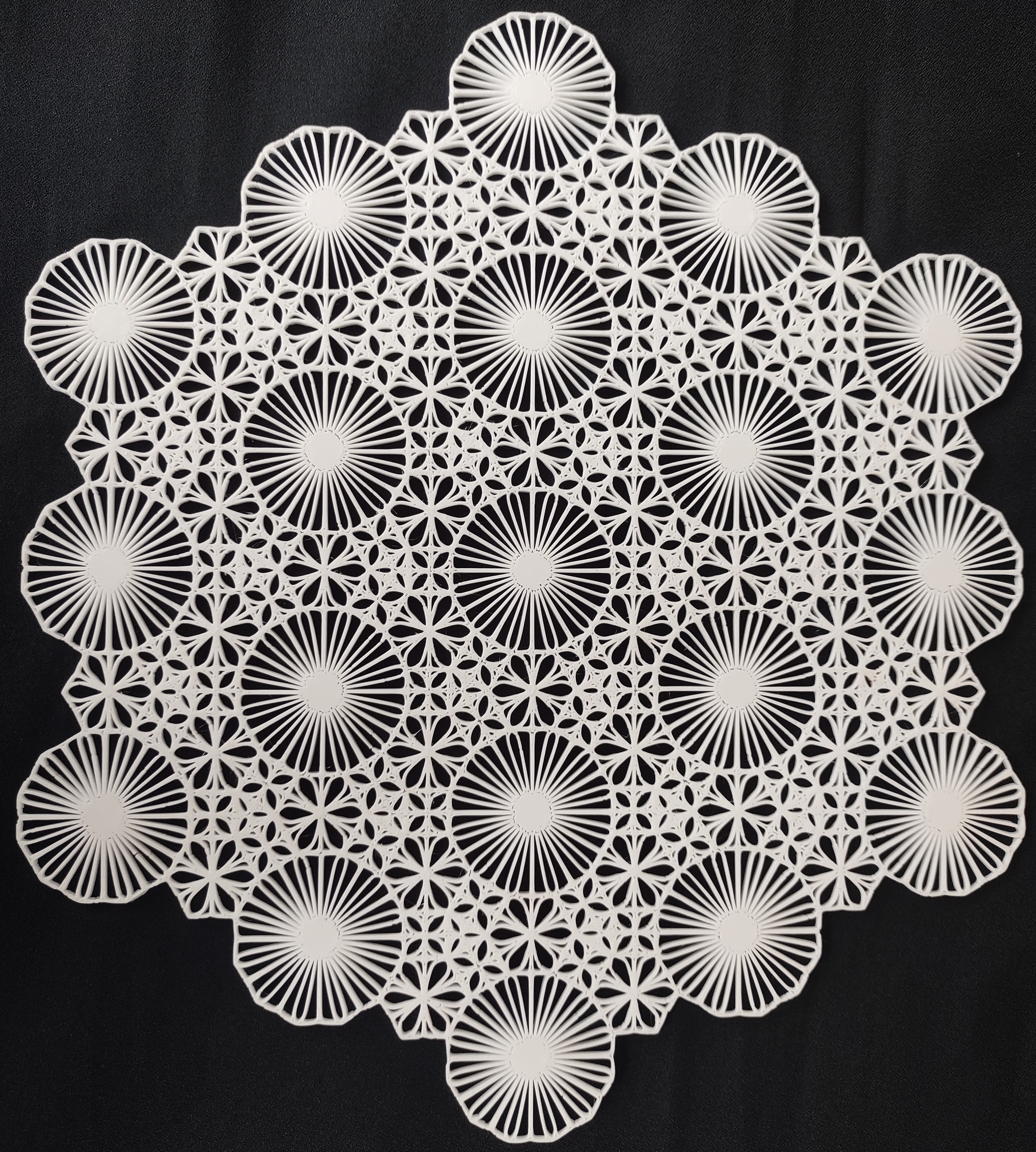}
\caption{\footnotesize An intricate tiling.} \label{fig-intricateTiling}
\end{center}
\end{subfigure}
\hspace*{0.01\linewidth}
\begin{subfigure}[t]{0.315\linewidth}
\begin{center}
\includegraphics[width=0.61\linewidth]{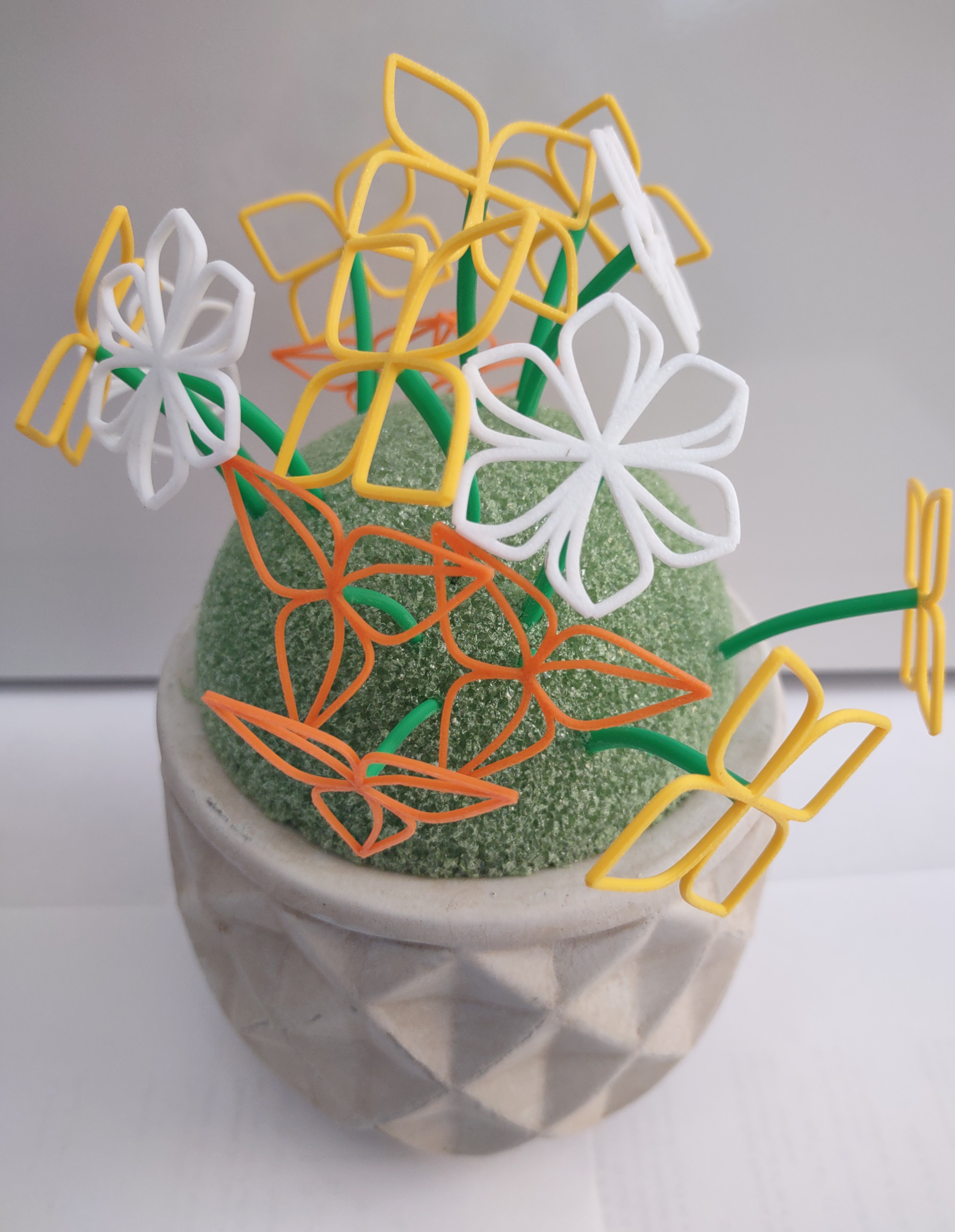}
\caption{\footnotesize A bouquet.} \label{fig-bouquet}
\end{center}
\end{subfigure}
\end{center}
\caption{FDM printed items.} 
    \label{fig-fdmPrints}
\end{figure}

\subsection*{Print Making with Cyanotype and SolarFast}

Print making with patterns and/or tilings of two-dimensional central paths
has become one of our favorite techniques. We use cyanotype and SolarFast, 
both of which are dyes that react with ultraviolet light to create color, 
to imprint patterns on paper and cloth. Cyanotyping is the original process
used to make blueprints, and it only gives shades of blue. SolarFast is 
similar but comes in a variety of colors.

The technique begins with a design like the one in Figure~\ref{fig-intricateTiling} 
--- we use a modern and accurate FDM printer. We note that the image processing 
software requires special settings to ensure fidelity to our mathematical calculations, 
for otherwise, the result is flawed with artifacts from aliasing. We then pre-stretch 
watercolor paper to reduce warping after the printing process. Colorless dye 
is then applied to the paper, and the 3D-printed pattern is placed over the dyed 
area. We subsequently expose the paper to sunlight for several minutes to create
color in the unblocked regions. The print is then washed with a special soap 
and is left to dry. The drying process ends with improved flatness if the paper 
is held in place as it dries. The results are, at least to the authors, visually stunning, 
see the previously published pieces in Figure~\ref{fig-SolarFast} as 
examples\footnote{Both images published in~\cite{interiart} and reused here with 
the authors' permission.}. A piece similar to Figure~\ref{fig-SolarFast1}, but with 
a solid, dark-blue background, is also now part of Rose-Hulman's permanent art collection.

\begin{figure}[ht]
\begin{center}
\begin{subfigure}[t]{0.45\linewidth}
\begin{center}
\includegraphics[width=0.6\linewidth]{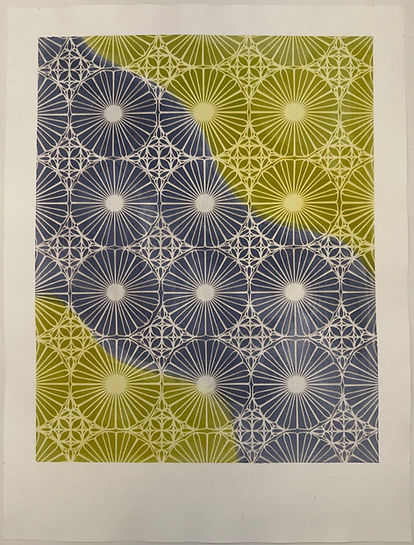}
\caption{\footnotesize Uniform Paths in a 3-4-3-12 Tiling with
    Blue Green Squiggle Background~\cite{interiart}.} 
    \label{fig-SolarFast1}
\end{center}
\end{subfigure}
\hspace*{0.01\linewidth}
\begin{subfigure}[t]{0.47\linewidth}
\begin{center}
\includegraphics[width=0.75\linewidth]{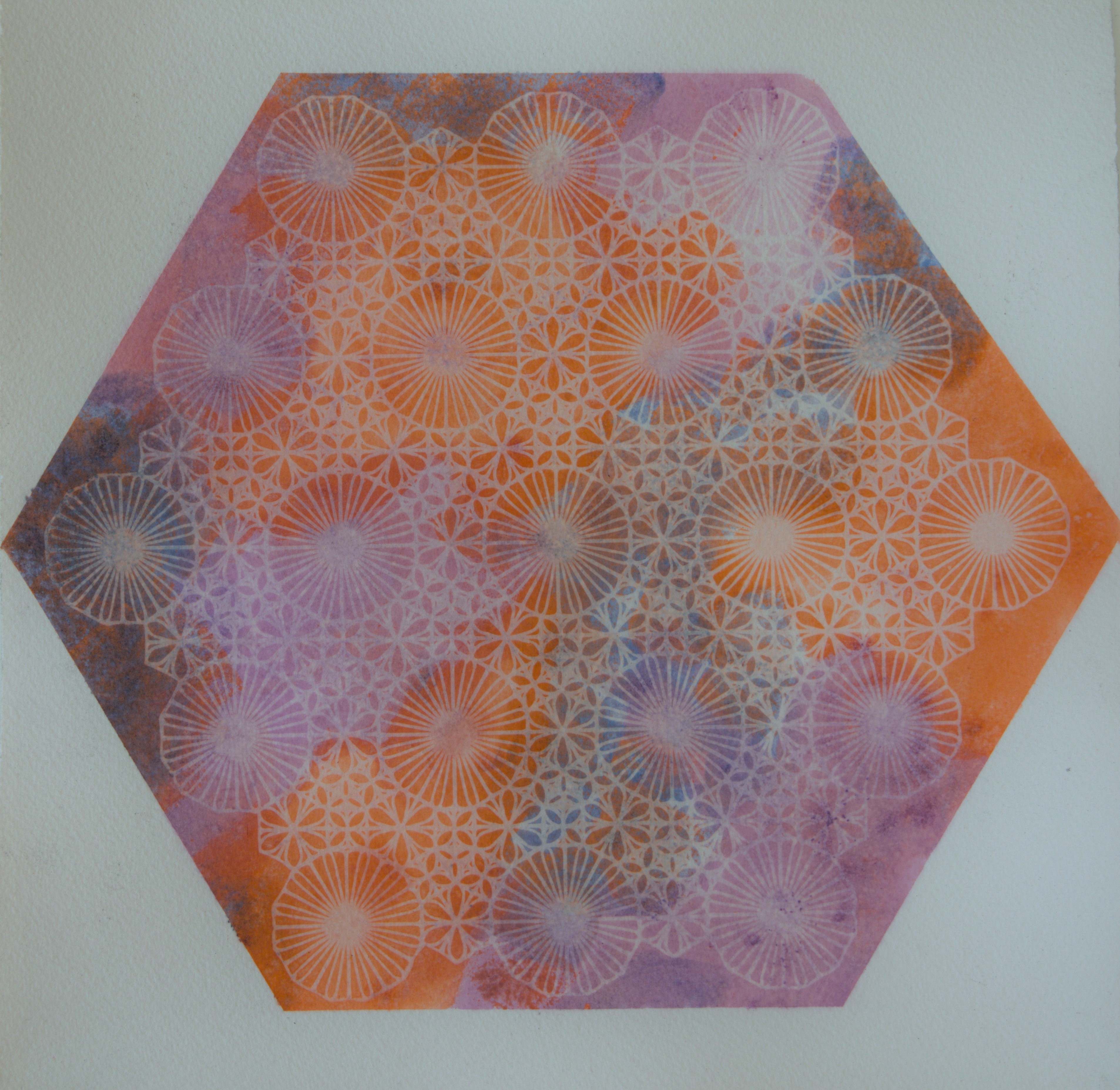}
\caption{\footnotesize Uniform Paths in a Truncated
    Trihexagonal Tiling with Nebular Background~\cite{interiart}.} 
    \label{fig-SolarFast2}
\end{center}
\end{subfigure}
\end{center}
\caption{Prints made with SolarFast.} 
    \label{fig-SolarFast}
\end{figure}

We used the SolarFast print making process to promote mathematics to middle school 
students during Rose-Hulman's 2025 Sonia Math Day. Students selected from
patterns after learning about central paths, and they then applied SolarFast
Dye to a $4$ inch by $6$ inch piece of watercolor paper. We then exposed and washed 
their prints, and everyone left with a unique piece of art to remember the day.
The collection of pieces is in Figure~\ref{fig-SoniaMath}. We have also used 
SolarFast dye to imprint t-shirts and other cloth items, and Figure~\ref{fig-tshirt}
shows the front and back of a shirt.

\begin{figure}[ht]
\begin{center}
\begin{subfigure}[t]{0.45\linewidth}
\begin{center}
\includegraphics[width=0.8\linewidth]{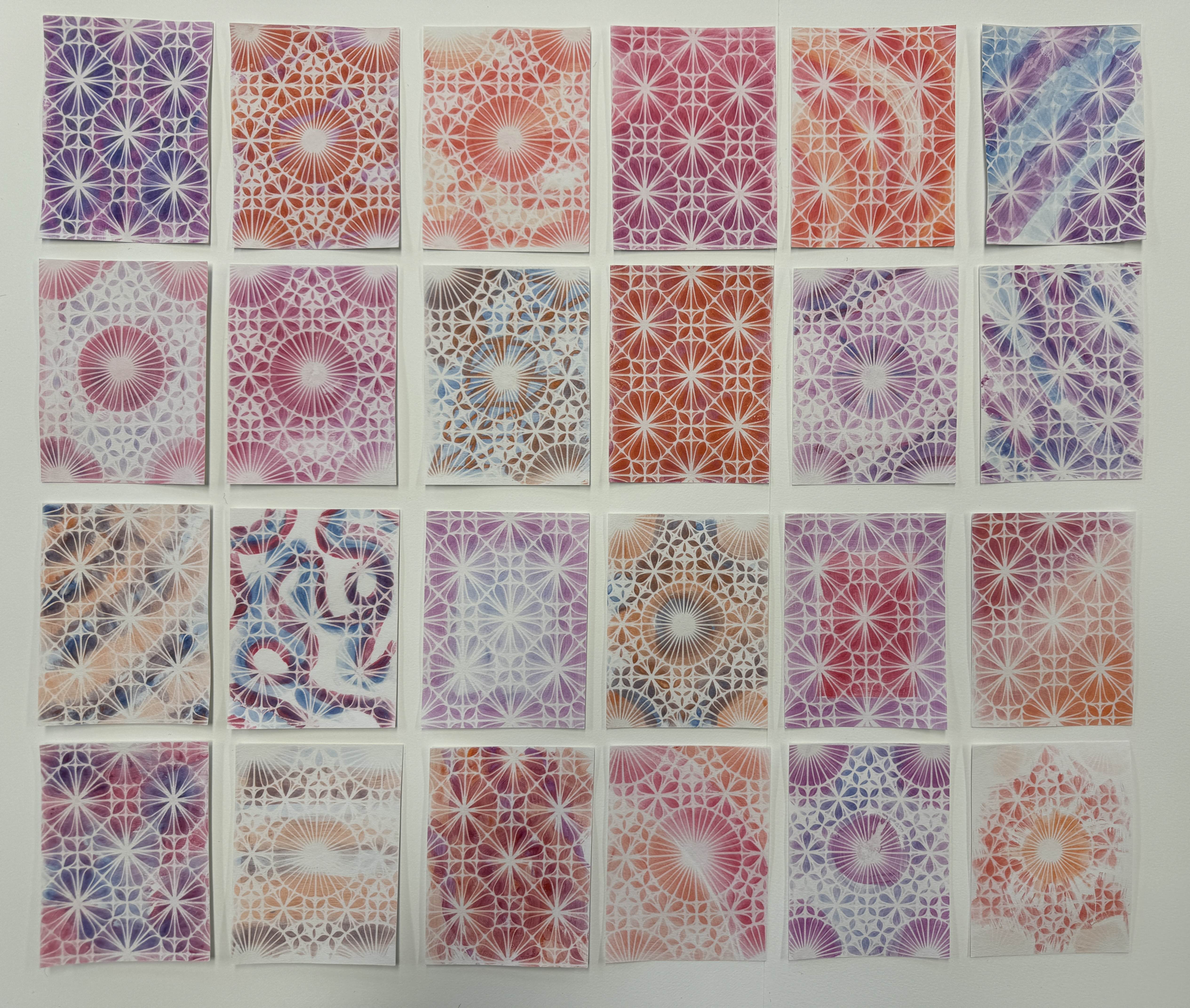}
\caption{\footnotesize Art projects from Rose-Hulman's
    2025 Sonia Math Day.} \label{fig-SoniaMath}
\end{center}
\end{subfigure}
\hspace*{0.01\linewidth}
\begin{subfigure}[t]{0.5\linewidth}
\begin{center}
\includegraphics[width=0.45\linewidth]{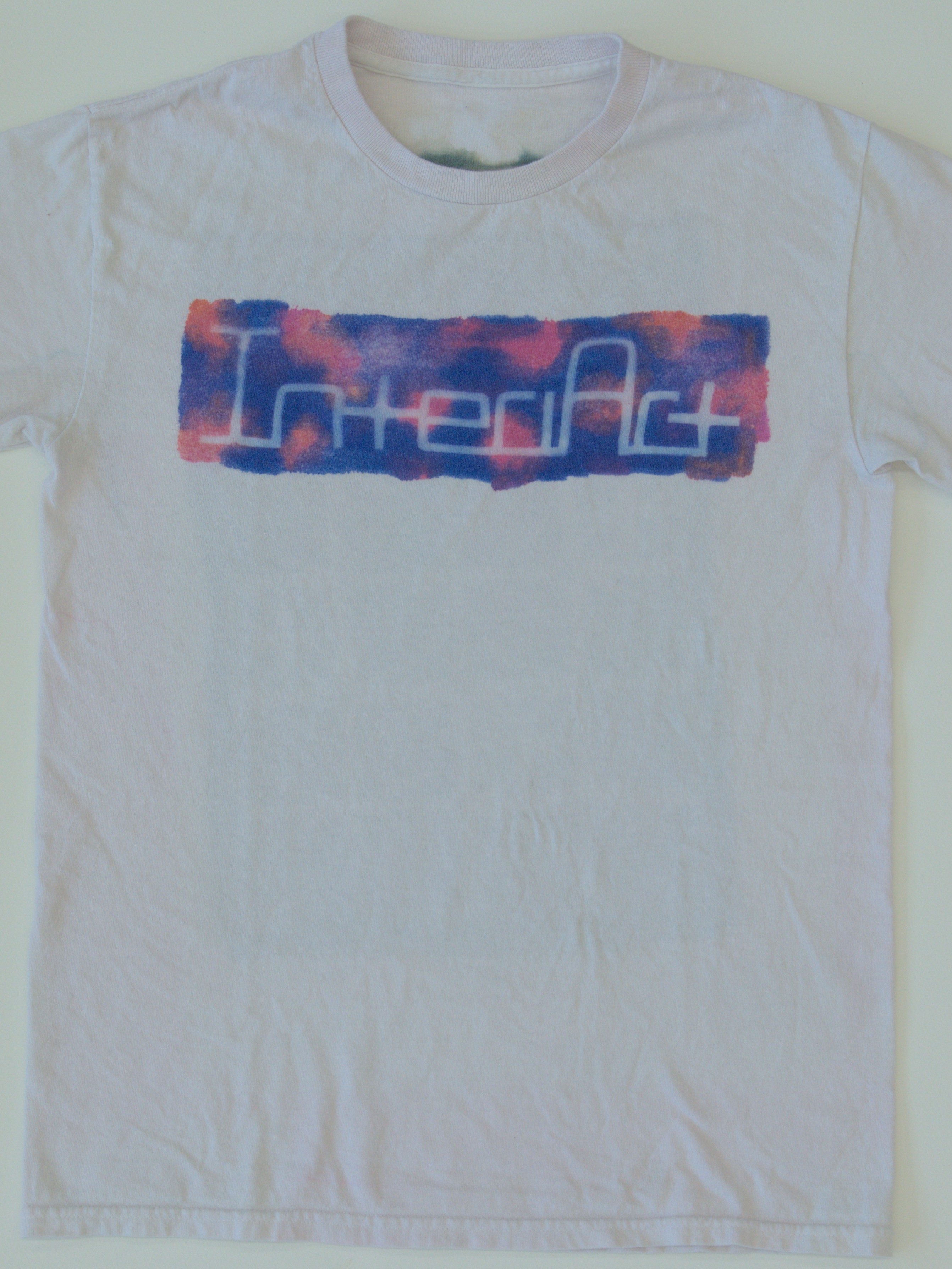}
\includegraphics[width=0.45\linewidth]{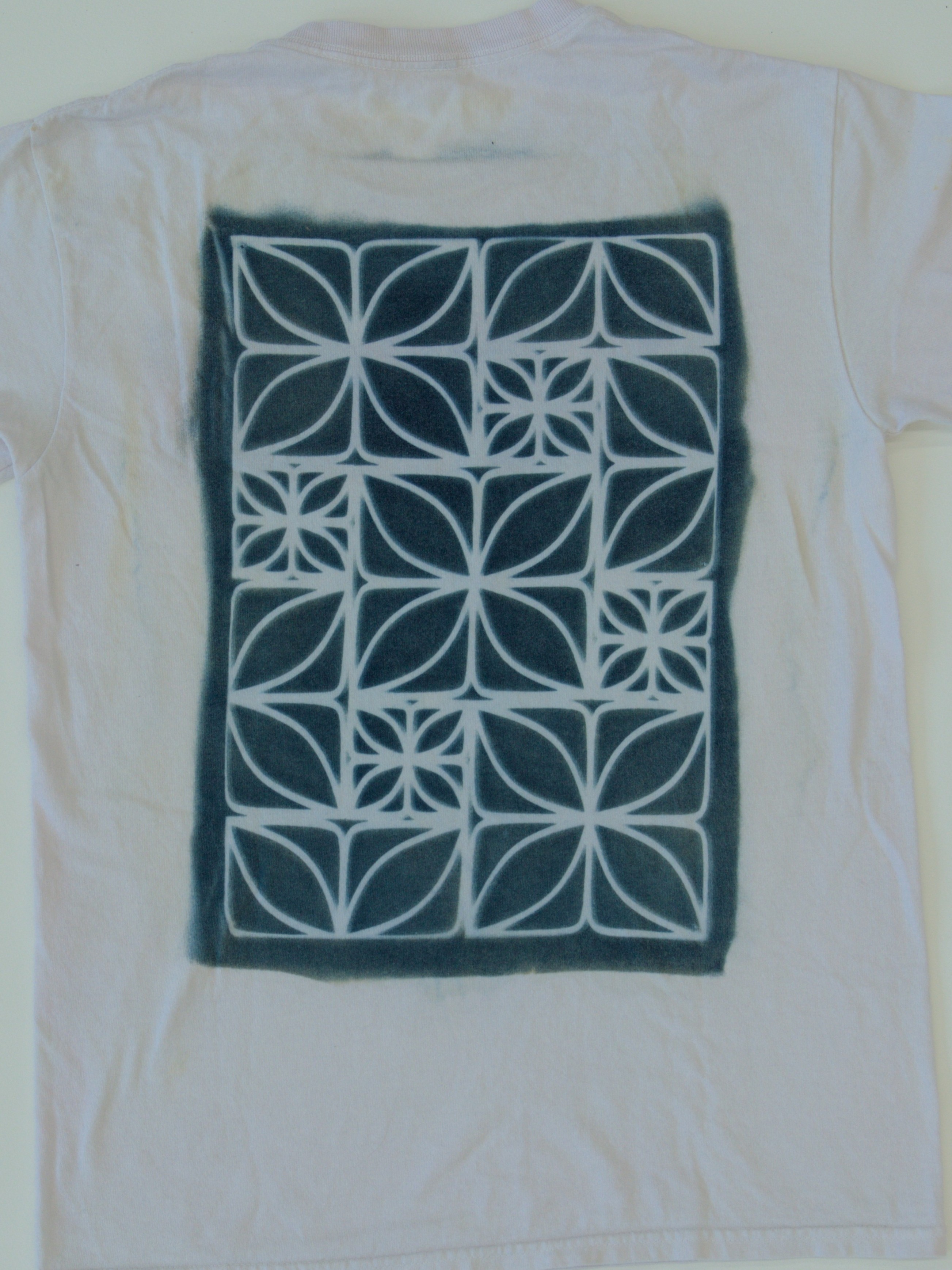}
\caption{\footnotesize T-shirt printing.} \label{fig-tshirt}
\end{center}
\end{subfigure}

\end{center}
\caption{SolarFast projects and a cloth example.} 
    \label{fig-SolarFastProjects}
\end{figure}

Combining prints and wood cuttings elevates the visual effect. These pieces
have a SolarFast print behind a laser-cut wood pattern, rendering a collection
of ``windows" through which the colored patterns are visible. Figure~\ref{fig-combined} 
contains two examples. The wood portion of the clock 
in Figure~\ref{fig-clock} has paths in a $12$-gon, with each leaf having four
paths. The red-orange SolarFast background of the clock is a tiling of paths in 
hexagons and triangles. The large wall piece in Figure~\ref{fig-bigWall} is over 
$10$ feet long and $4$ feet tall. The wood is maple veneered, $1/4$-inch plywood finished 
with linseed oil, and the piece has  $26$ triangular patterns, $12$ square patterns, 
and one $12$-gon pattern. The colorful SolarFast backdrop is created from tilings 
with paths in $3$-, $4$-, $6-$, $8-$, and $12$-gons. This piece of art is part of 
Rose-Hulman's permanent collection, and a discussion of the artistic techniques and
efforts is in the attached supplement.

\begin{figure}[ht]
\begin{center}
\begin{subfigure}[t]{0.45\linewidth}
\begin{center}
\includegraphics[width=0.6\linewidth]{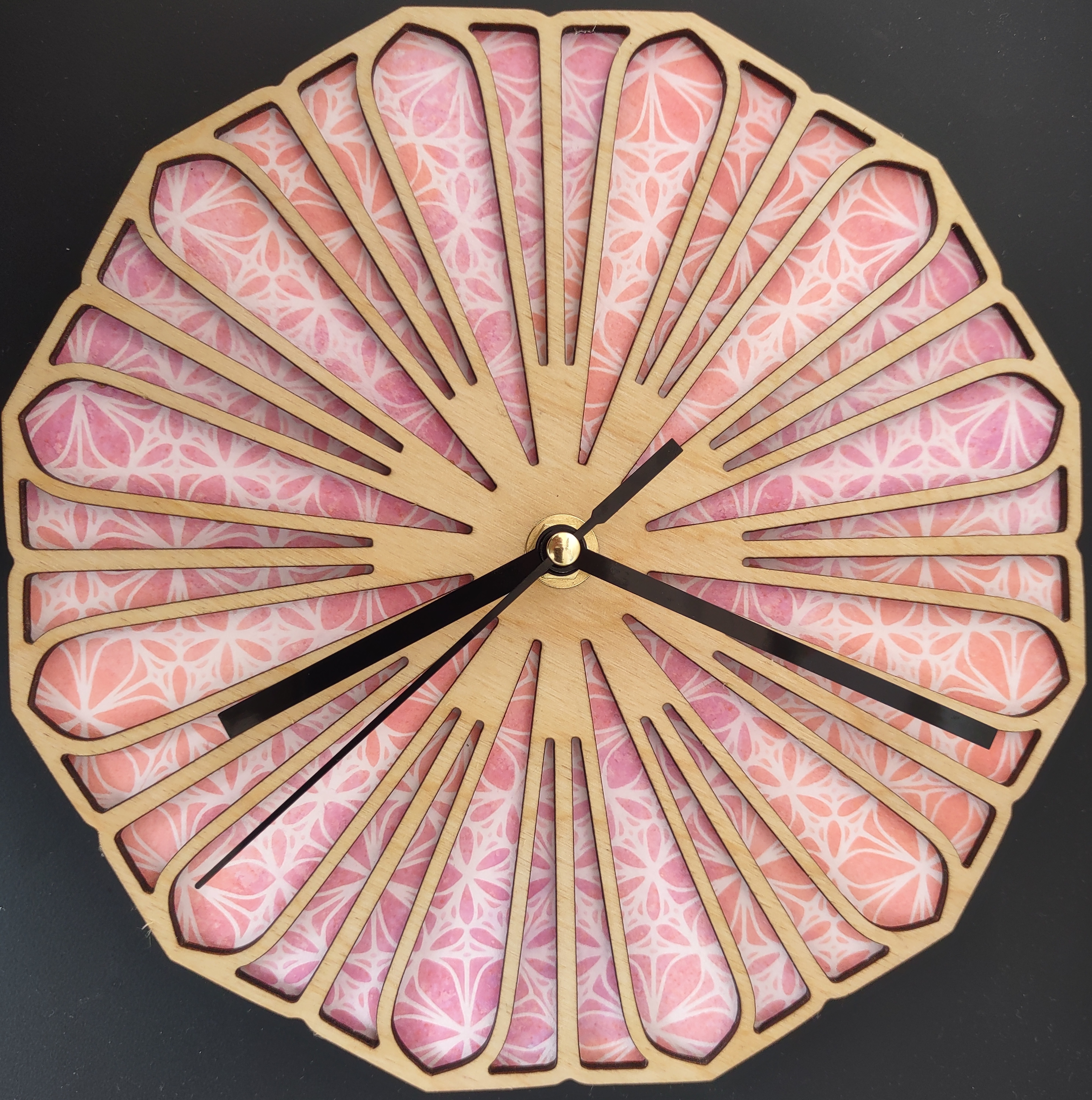}
\caption{\footnotesize A wall clock.} \label{fig-clock}
\end{center}
\end{subfigure}
\hspace*{0.01\linewidth}
\begin{subfigure}[t]{0.45\linewidth}
\begin{center}
\includegraphics[width=0.68\linewidth]{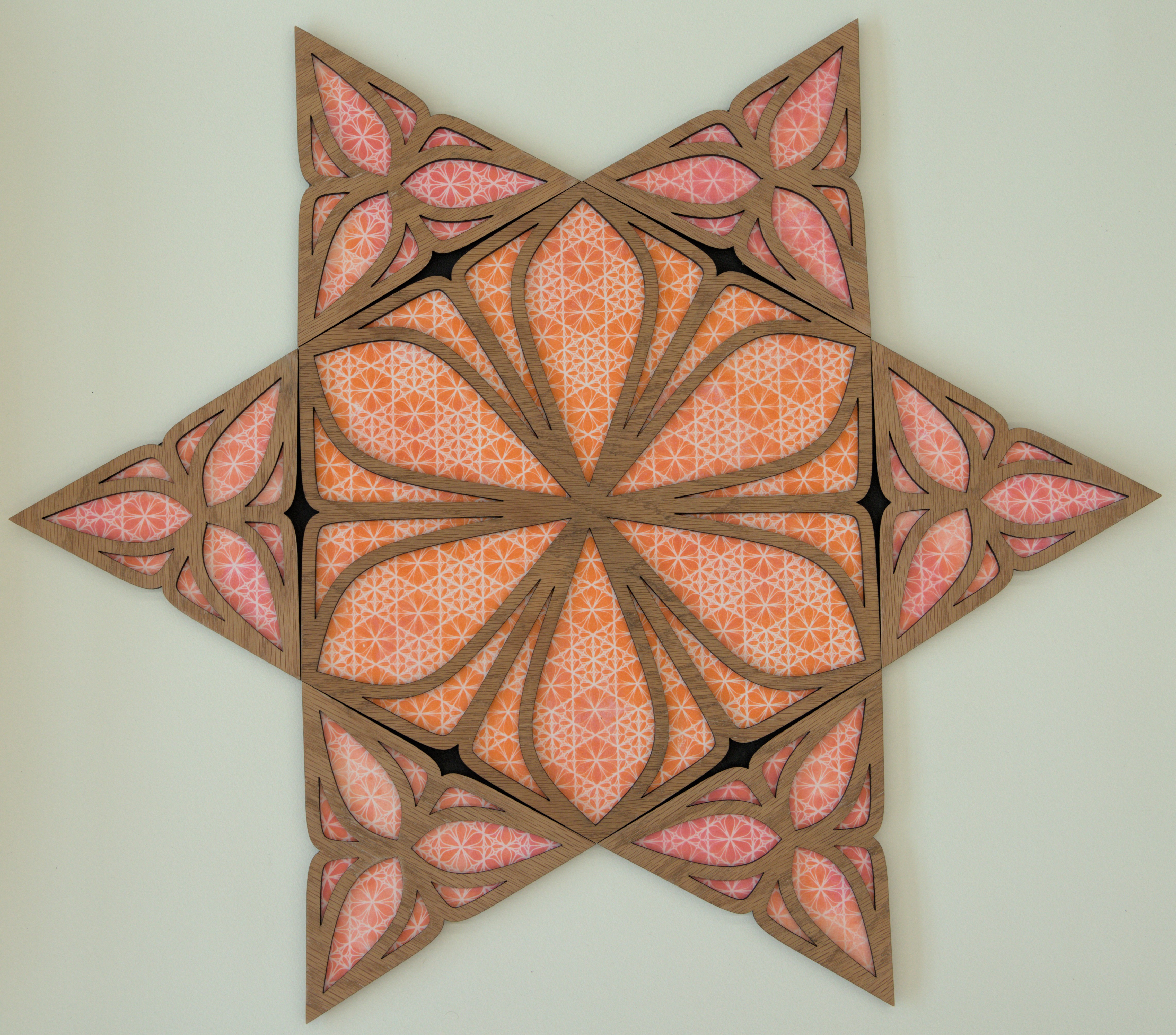}
\caption{\footnotesize A medium wall piece.} \label{fig-bigWall}
\end{center}
\vspace*{6pt}
\end{subfigure}
\begin{subfigure}[t]{0.95\linewidth}
\begin{center}
\includegraphics[width=0.8\linewidth]{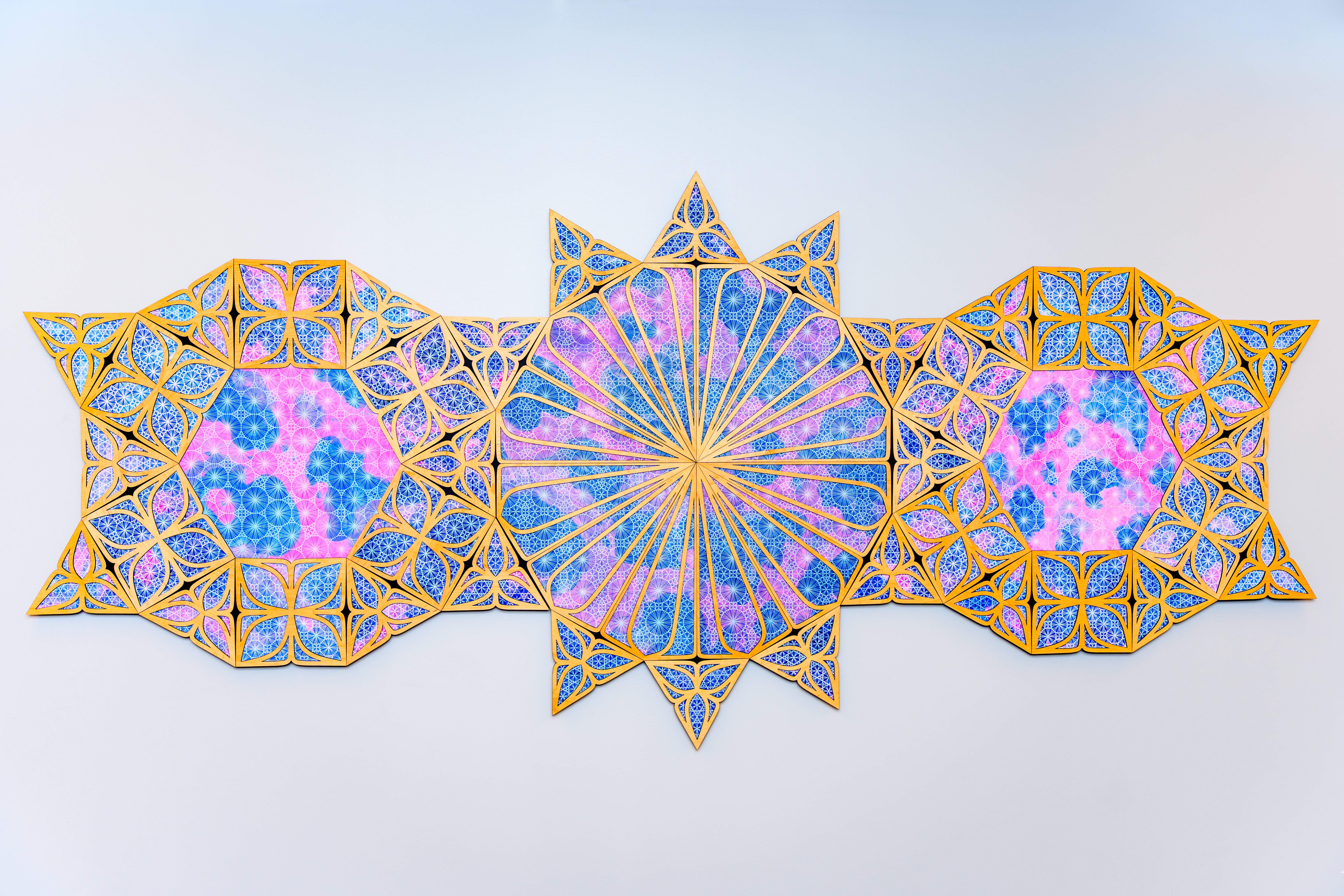}
\caption{\footnotesize A large wall piece.} \label{fig-bigWall}
\end{center}
\end{subfigure}
\end{center}
\caption{Pieces that combine SolarFast print making with wood
    cuttings.}  \label{fig-combined}
\end{figure}

\subsection{Three-Dimensional Art} \label{sec-3Dart}

Three-dimensional objects are visually possible as digital images, but many
are difficult to physically produce without special consideration and equipment. 
The most immediate issue is that FDM printers limit structure because they do not 
naturally accommodate overhanging features without additional, awkward, and unwanted 
vertical supports. This restriction limits creating three-dimensional renderings of paths in a
polytope, such as a Platonic solid, with standard FDM printers. We have had success 
with selective laser sintering (SLS) printing, but we postpone that discussion and 
focus momentarily on three-dimensional items possible with FDM printers.

A 3D printer requires a standard triangle language (stl) file that approximates the
surface of the object being printed, but curves in three-dimensional space
have no surface. Moreover, there are no standard python packages that generate printable stl 
files from parametrized curves, and as such, we created a custom package to generate printable 
files. The mathematical routine is simple in concept, as it essentially turns each path into 
a spaghetti-like object whose surface is then approximated by a triangulation, which is then 3D 
printable. We accomplish this task by identifying a collection of points on the path, computing 
a unit normal vector to the path at that point, and then placing points on a circle that is centered at 
the point on the path and that passes through the terminal end of a scaled version of the normal 
vector. This process ``extrudes'' a circle along the central path, with the result being a 
collection of points on the surface of a thickened version of the path, from which we create 
a triangulation for printing. 

\subsubsection*{Daisies and Thistles}

Daisies have relatively flat petal arrangements, which makes them easy to imitate
with two-dimensional patterns in $k$-gons. We adjoin a three-dimensional central path as stem to 
complete the facsimile. The simplicity of this design allows us to print daisies 
with an FDM printer so long as the central path forming the stem has limited overhang/bend.
Another option is a thistle. We mimic the spiky nature by setting $G(x) = x_1^2 + x_2^2 - 1$
to get a unit disc. Solving~\eqref{eq-necSuffCond} shows that
\[
\overline{\bP(G(x),c)} = \left\{ x(\mu) =
    \frac{\sqrt{\mu^2 + \| c \|^2} - \mu}{\|c\|^2} \, c \; : \: \mu > 0 \right\}
    \bigcup \left\{ 0, \, \frac{c}{\|c\|} \right\},
\]
and hence, central paths in a disc are (unsurprisingly) straight lines. We comment
that this fact is true for any two-norm unit sphere in higher dimensions. 
Figure~\ref{fig-daisyAndThistle} displays FDM printed items.

\begin{figure}[ht]
\begin{center}
\begin{subfigure}[t]{0.28\linewidth}
\begin{center}
\includegraphics[width=0.6\linewidth]{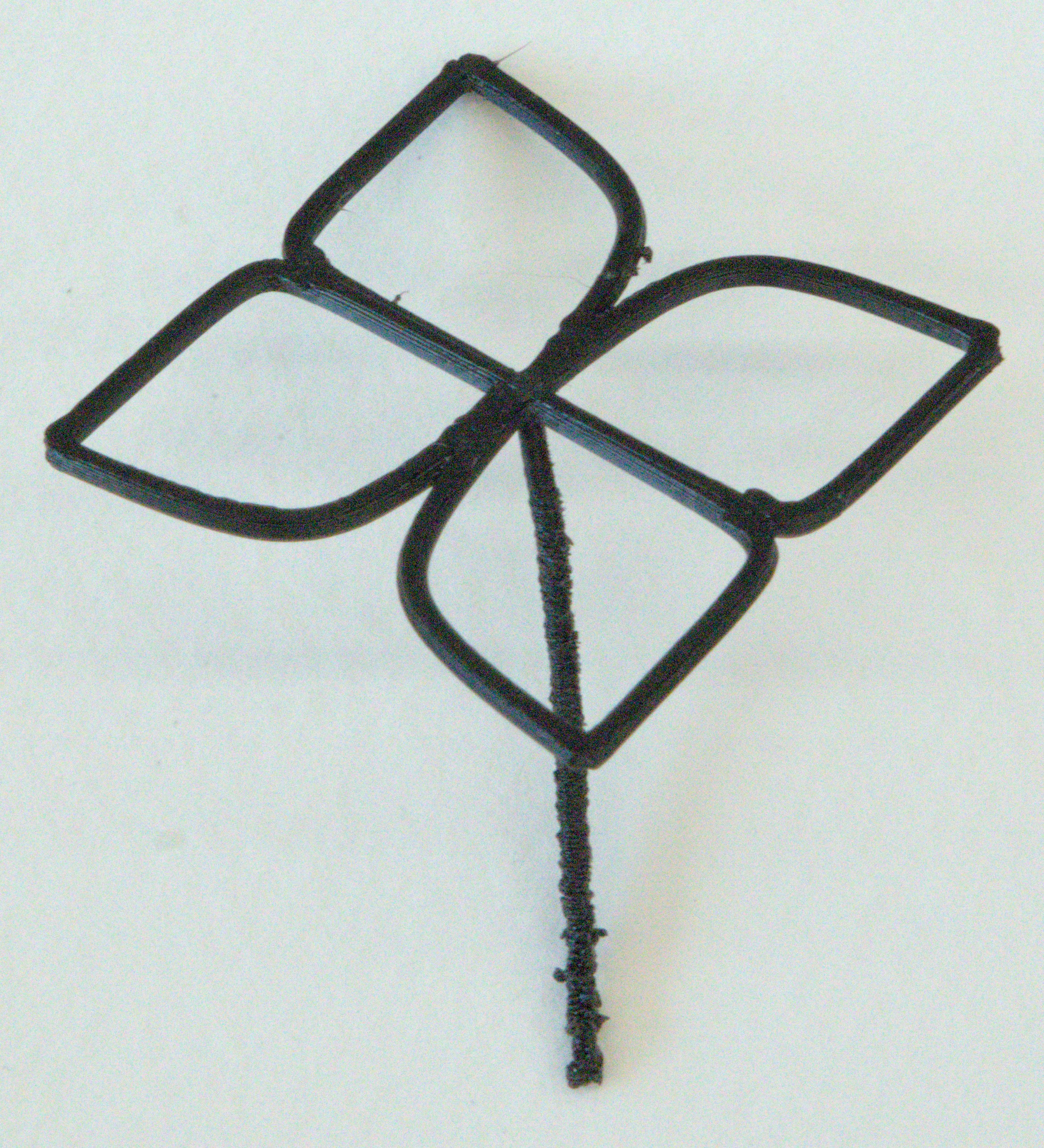}
\caption{\footnotesize FDM daisy.} \label{fig-daisyPrint}
\end{center}
\end{subfigure}
\hspace*{0.05\linewidth}
\begin{subfigure}[t]{0.28\linewidth}
\begin{center}
\includegraphics[width=0.685\linewidth]{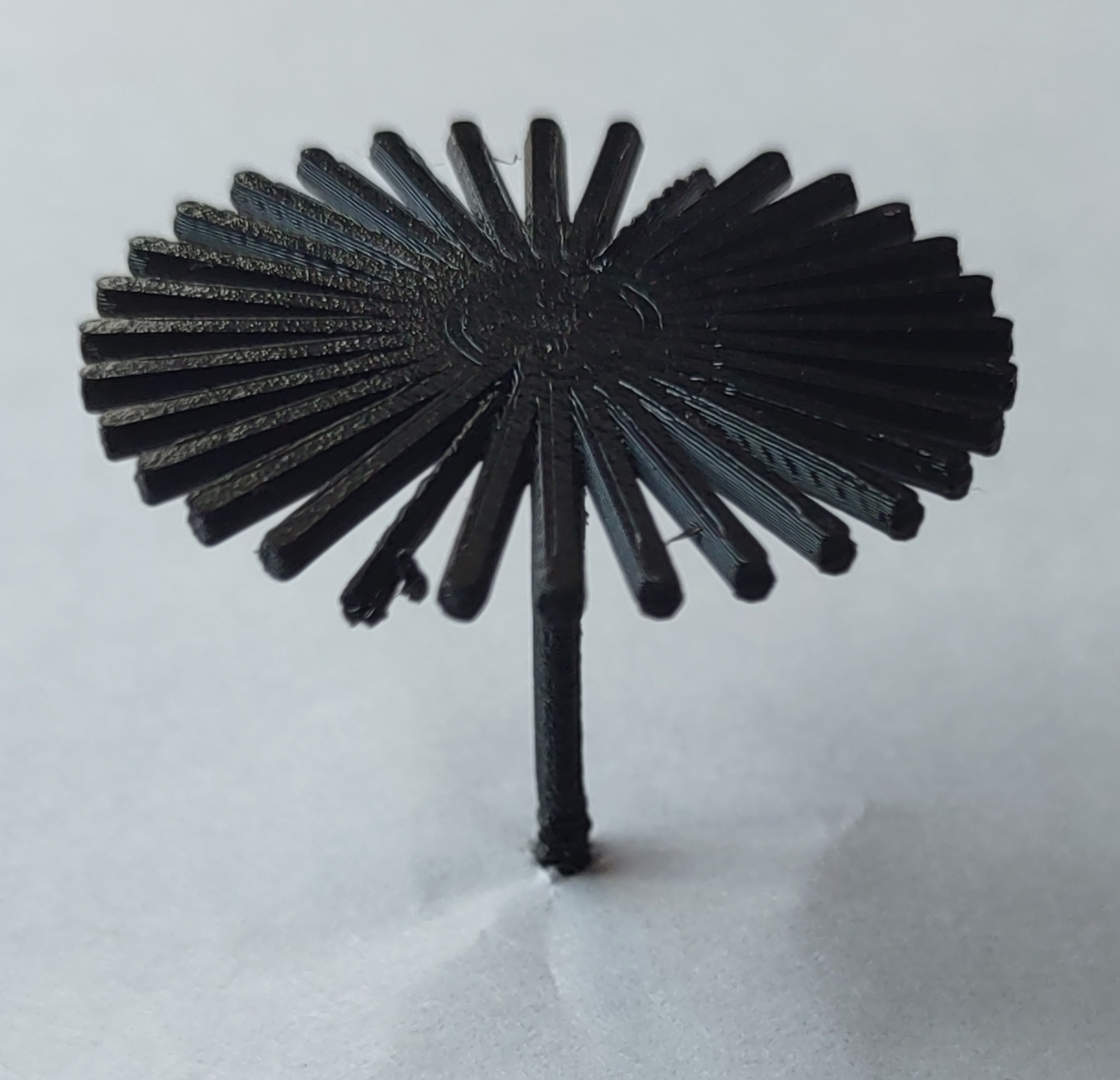}
\caption{\footnotesize FDM thistle.} \label{fig-thistlePrint}
\end{center}
\end{subfigure}
\end{center}
\caption{FDM prints of a daisy and a thistle.}  \label{fig-daisyAndThistle}
\end{figure}

\subsubsection*{SLS Printing}

Selective laser sintering (SLS) printing gives a significant advantage over
other three-dimensional printing modalities because it forms a mold of the item 
as it is being printed, and hence, there is no need for additional supports. SLS printers
build an item in layers similar to other printing paradigms, but SLS printing 
deposits a uniformly thin layer of polyamide material that is then fused with a 
laser to create a horizontal slice of the item(s) being printed. The unused 
polyamide material remains in place until the print is complete, creating 
a mold that supports the object. So SLS printers can print space-curves from stl files 
that approximate the curves. An SLS printer was unfortunately 
outside our funding ability, but we were able to pay a third party to create 
a single print. We selected random paths in a cube as our first test case because 
these images started our project. Figure~\ref{fig-SLS} 
shows the original three-dimensional image and the SLS print. We comment that
the SLS print is not cubic although the stl file indicated that it should have
been. We suspect that either the printer or the vendor reduced the vertical 
direction due to machine limitations or production restrictions.

\begin{figure}[ht]
\begin{center}
\begin{subfigure}[t]{0.45\linewidth}
\begin{center}
\includegraphics[width=0.7\linewidth]{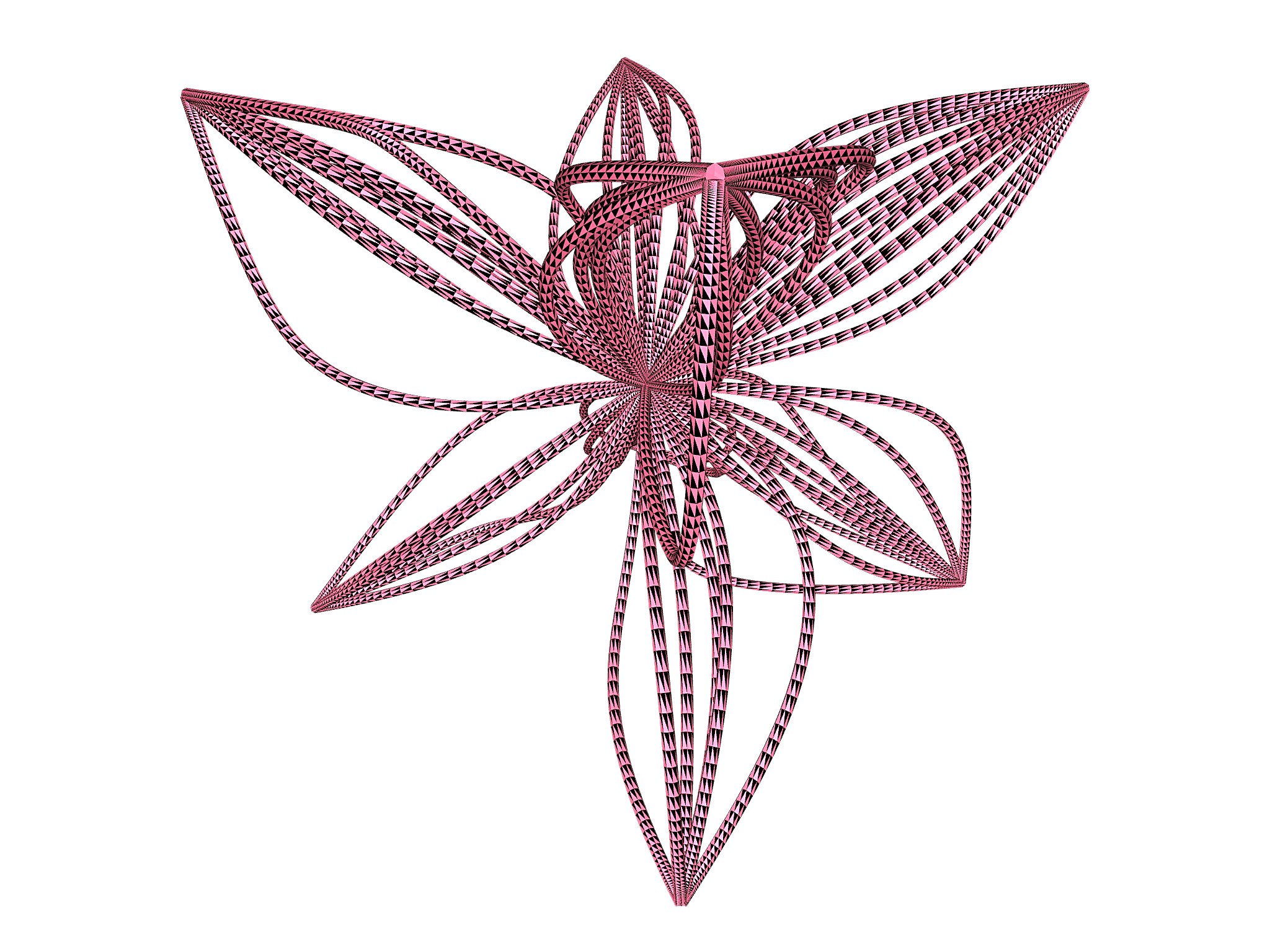}
\caption{\footnotesize An stl image of a cubic flower.} \label{fig-cubicFlowerImage}
\end{center}
\end{subfigure}
\hspace*{0.01\linewidth}
\begin{subfigure}[t]{0.45\linewidth}
\begin{center}
\includegraphics[width=0.7\linewidth]{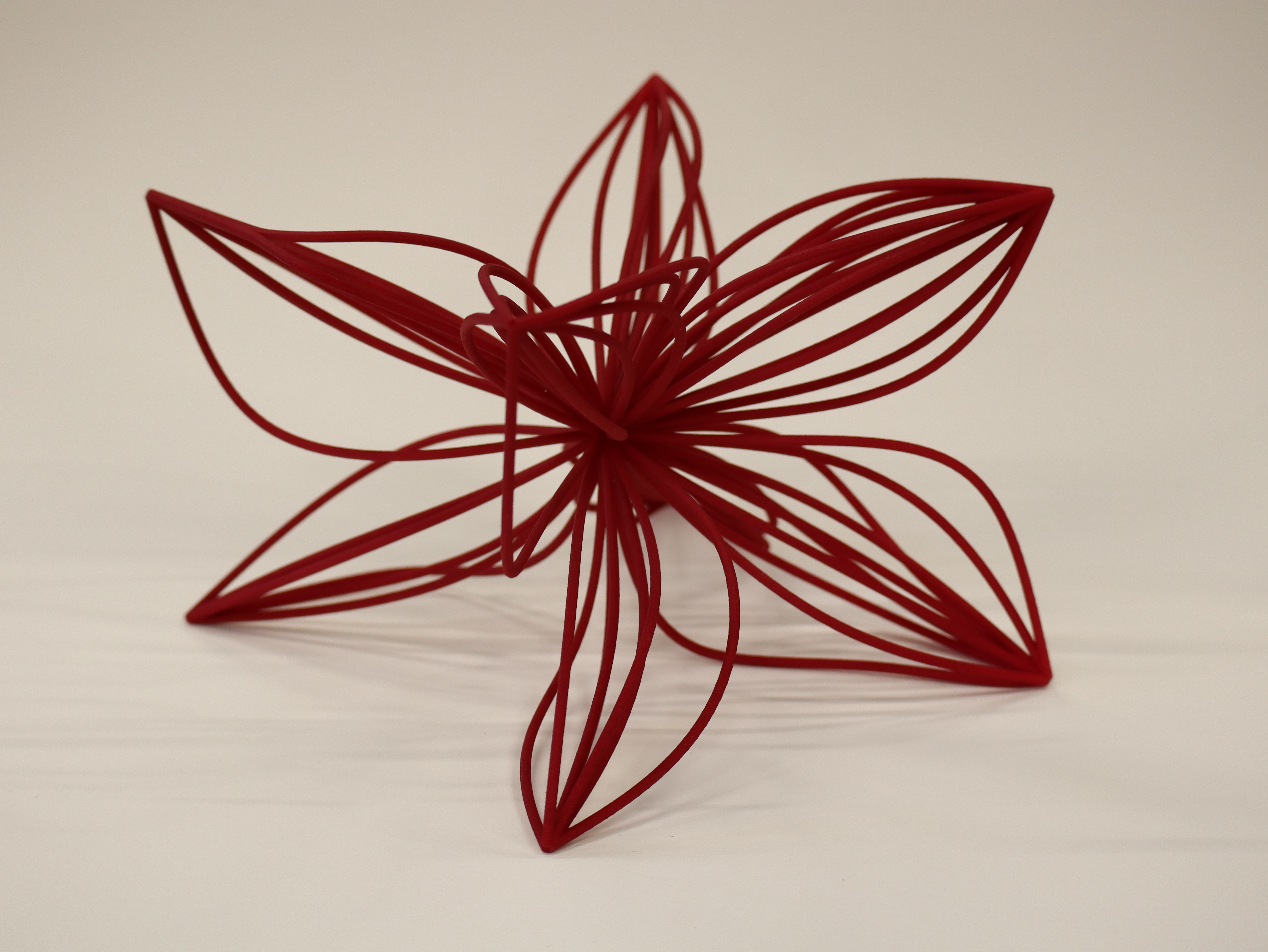}
\caption{\footnotesize An SLS printed cubic flower.} \label{fig-cubicFlowerPrint}
\end{center}
\end{subfigure}
\end{center}
\caption{An example of SLS printing.}  \label{fig-SLS}
\end{figure}

\section{Conclusion and Future Goals} \label{sec-conclusion}

We promote that the central path is a thing of beauty and is worthy of
artistic exploration. Its revolutionary impact on the field of optimization
is, in itself, an esoteric and rigorous work of art, but we have now
shown that the central path is also a unique `brush stroke' with which we
can create items of physical beauty. The mathematics of the central
path controls these two- and three-dimensional constructs, and we
manifest brush strokes by changing mathematical models. So in
this project, artistic modeling is mathematical modeling, and the bridge
between the two paradigms is a significant amount of computing.

Our primary goal for the future is to advance our ability with SLS
printing, and in particular, we hope to create a large optimization garden full 
of flower-like objects. Much of the mathematical and computational effort is 
in place, and we have, for instance, developed code to generate paths in the 
five Platonic solids. The authors have found this effort to be a unique 
mathematical enterprise because, while optimizers and geometers have long known 
how to express polytopes in terms of either facet inequalities or as convex 
combinations of vertices, going from one to the other has been a chore even 
though the mathematical relationships are well understood. Descriptions of 
the Platonic solids are commonly expressed in terms of standardized coordinates, 
i.e. in terms of vertices, but our mathematical models require facet inequalities. 
We initially formulated a linear combinatorial problem whose solution was a normal 
to a facet, and while this problem worked well on the tetrahedron and the 
cube, it proved insufficient or impossible to solve on the octahedron, the dodecahedron, 
and the icosahedron. So we moved to a search over subsets of the vertices to identify 
which subsets defined facets, from which we could then calculate a facet normal. This 
proved efficient on all five solids.

\begin{figure}[ht]
\begin{center}
\begin{subfigure}[t]{0.45\linewidth}
\begin{center}
\includegraphics[width=0.8\linewidth]{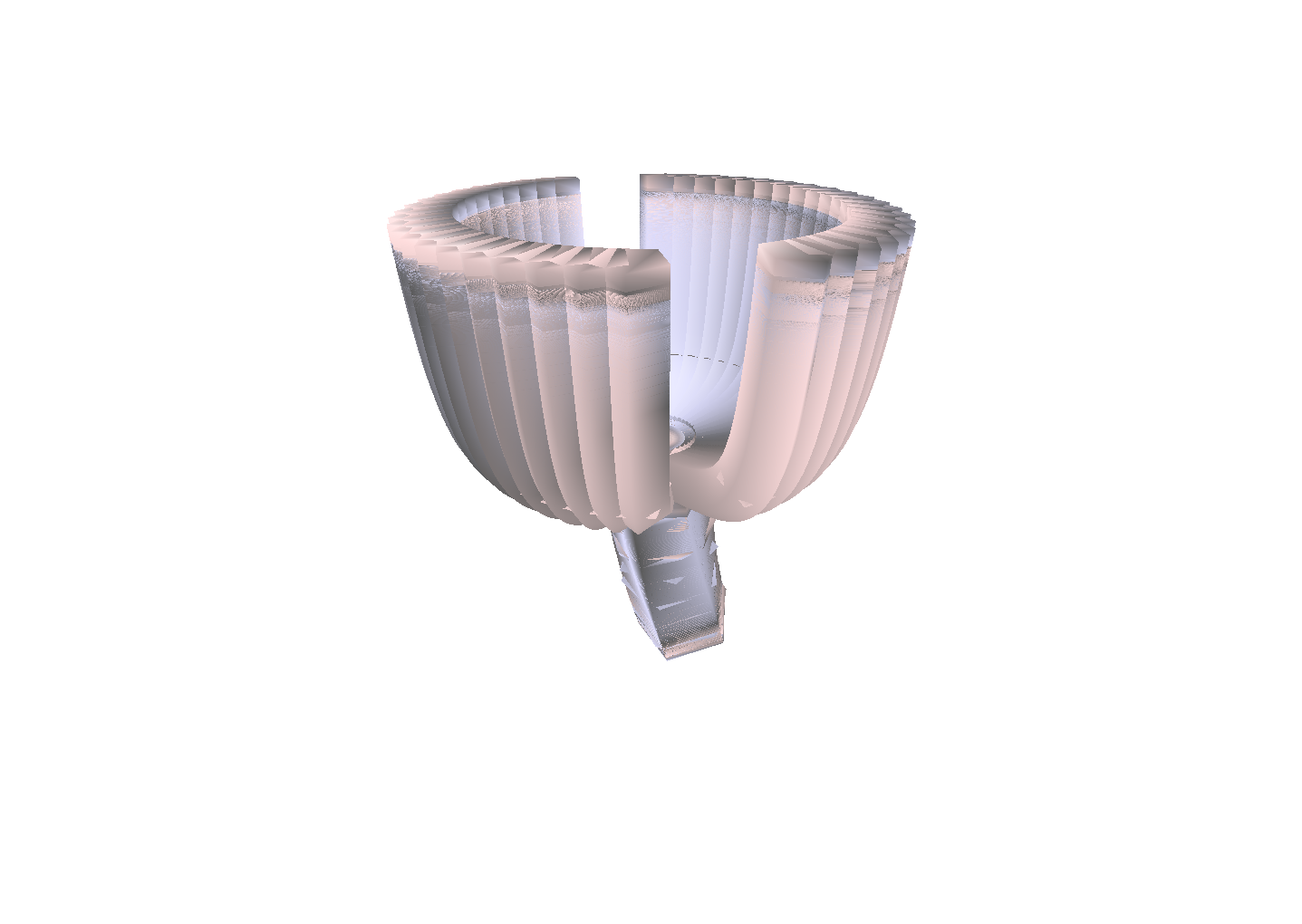}
\caption{\footnotesize An stl image of a tulip.} \label{fig-tulip}
\end{center}
\end{subfigure}
\hspace*{0.01\linewidth}
\begin{subfigure}[t]{0.45\linewidth}
\begin{center}
\includegraphics[width=0.9\linewidth]{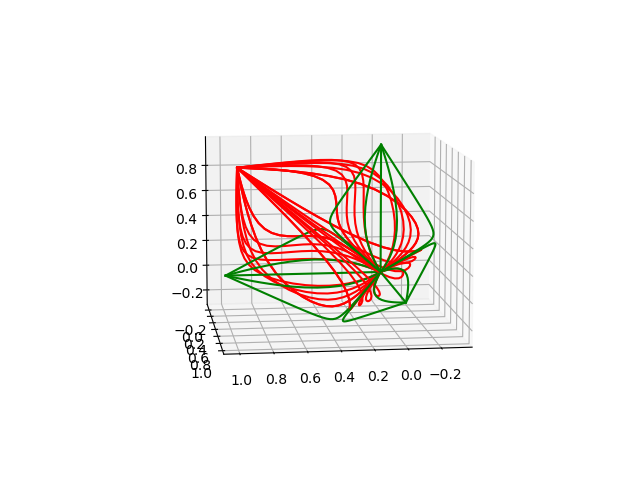}
\caption{\footnotesize A rosebud.} \label{fig-roseBud}
\end{center}
\end{subfigure}
\end{center}
\caption{Digital renderings for potential SLS printing.}  \label{fig-SLSpotential}
\end{figure}

The SLS print in Figure~\ref{fig-cubicFlowerPrint} illustrates possibility,
but this flower is surely, and wonderfully to the authors, a bit Burtonesque.
The tulip and rosebud images in Figure~\ref{fig-SLSpotential} are
more realistic, but we are not yet able to print these items. Theorem~\ref{thm-cylinder}
establishes that we can create a surface of paths for the petal of a tulip,
and we use our previous thickening technique on a discrete collection of paths to
emulate a petal. The result is an stl file like the one in Figure~\ref{fig-tulip}.
This technique unfortunately renders a file with an excessively large triangulation due
to the numerous intersections among the paths, and our printing software is unable
to slice the object for printing. We hope to overcome this issue soon by extending our 
triangulation software to avoid overlaps and intersections.
The rosebud in~\ref{fig-roseBud} is created by paths in a tetrahedron, and while our thickening 
technique should work similar to the item in~\ref{fig-cubicFlowerPrint}, we have not yet
generated stl files because we do not have access to an SLS printer. There is much more to explore, 
and in particular, we plan to learn which shapes lend themselves to SLS printing, what materials 
and finishes work well and give lasting results, and which computational adjustments ensure/advance
the mathematical fidelity of the printed items. \\\

\noindent{\bf Acknowledgments:} The authors thank Ms.~Soully Abas, Associate
Professor of Art, and Ms.~Christy Brinkman-Robertson, Art Curator, both at the
Rose-Hulman Institute of Technology. Their support, encouragement, and guidance
have been invaluable.

\bibliographystyle{plain}
\bibliography{citations}

\begin{thebibliography}{10}

\bibitem{caron02}
R.~Caron, H.~Greenberg, and A.~Holder.
\newblock Analytic centers and repelling inequalities.
\newblock {\em European Journal of Operational Research}, 143(2):268--290,
  2002.

\bibitem{deza06}
A.~Deza, E.~Nematollahi, R.~Peyghami, and T.~Terlaky.
\newblock The central path visits all the vertices of the {K}lee-{M}inty cube.
\newblock {\em Optimisation Methods and Software}, 21(5):851--865, 2006.

\bibitem{dikin67}
I.~Dikin.
\newblock Iterative solution of problems of linear and quadratic programming.
\newblock {\em Soviet Mathematics. Doklady}, 8:674--675, 1967.

\bibitem{fiacco68}
A.~Fiacco and G.~McCormick.
\newblock {\em Nonlinear Programming: Sequential Unconstrained Minimization
  Techniques}.
\newblock Wiley, 1968.
\newblock reprinted, Classics Applied Mathematics, SIAM, Philadelphia, 1990.

\bibitem{forsgren02}
A.~Forsgren, P.~Gill, and M.~Wright.
\newblock Interior methods for nonlinear optimization.
\newblock {\em SIAM Review}, 44(4):525--597, 2002.

\bibitem{frisch1955}
R.~Frisch.
\newblock The logarthmic potential method for solving linear programming
  problems.
\newblock {\em Memorandum from Institute of Economics}, 1955.

\bibitem{gill86}
P.~Gill, W.~Murray, M.~Saunders, J.~Tomlin, and M.~Wright.
\newblock On projected {N}ewton barrier methods for linear programming and an
  equivalence to {K}armarkar’s projective method.
\newblock {\em Mathematical Programming}, 36(2):183--209, 1986.

\bibitem{interiart}
C.~Tasik \& InteriArt~Research Group.
\newblock {\em Ink: Rose-Hulman's Literary and Visual Arts Magazine}.
\newblock {\em Ink}, vol. 23, 2025.
\newblock www.rose-ink.wixsite.com/home.

\bibitem{halicka99}
M.~Halick{\'a}.
\newblock Analytical properties of the central path at boundary point in linear
  programming.
\newblock {\em Mathematical Programming}, 84(2), 1999.

\bibitem{holder04}
A.~Holder.
\newblock Simultaneous data perturbations and analytic center convergence.
\newblock {\em SIAM Journal on Optimization}, 14(3):841--868, 2004.

\bibitem{holder00}
A.~Holder and R.~Caron.
\newblock Uniform bounds on the limiting and marginal derivatives of the
  analytic center solution over a set of normalized weights.
\newblock {\em Operations Research Letters}, 26(2):49--54, 2000.

\bibitem{huard64}
P.~Huard.
\newblock Resolution des programmes mathematiques par la methode des centres.
\newblock Technical report, Note E.D.F. HR 5690, 1964.

\bibitem{huard67}
P.~Huard.
\newblock Resolution of mathematical programming with nonlinear constraints by
  the method of centers.
\newblock In {\em Nonlinear Programming}, pages 207--219. North Holland, 1967.

\bibitem{jarre88}
F.~Jarre, G.~Sonnevend, and J.~Stoer.
\newblock An implementation of the method of analytic centers.
\newblock In A.~Bensoussan and J.~Lions, editors, {\em Analysis and
  Optimization of Systems}, pages 295--308, Berlin, Heidelberg, 1988. Springer
  Berlin Heidelberg.

\bibitem{karmarkar84b}
N.~Karmarkar.
\newblock A new polynomial-time algorithm for linear programming.
\newblock {\em Combinatorica}, 4:373--395, 1984.

\bibitem{karmarkar84a}
N.~Karmarkar.
\newblock A new polynomial-time algorithm for linear programming.
\newblock In {\em Proceedings of the sixteenth annual ACM symposium on Theory
  of computing}, pages 302--311, 1984.

\bibitem{khachiyan79}
L.~Khachiyan.
\newblock A polynomial algorithm in linear programming.
\newblock {\em Soviet Mathematics. Doklady}, 20:191--194, 1979.

\bibitem{klee72}
V.~Klee and G.~Minty.
\newblock How good is the simplex algorithm?
\newblock In O.~Shisha, editor, {\em Inequalities III}, page 159–175.
  Academic Press, 1972.

\bibitem{deloera12}
J.~De Loera, B.~Sturmfels, and C.~Vinzant.
\newblock The central curve in linear programming.
\newblock {\em Foundations of Computational Mathematics}, 12(4):509--540, 2012.

\bibitem{nematollahi08}
E.~Nematollahi and T.~Terlaky.
\newblock A simpler and tighter redundant {K}lee-{M}inty construction.
\newblock {\em Optimization Letters}, 2:403--414, 2008.

\bibitem{roos2005}
C.~Roos, T.~Terlaky, and J-Ph Vial.
\newblock {\em Interior point methods for linear optimization}.
\newblock Springer Science \& Business Media, 2005.

\bibitem{mpg}
J.~Sauppe, editor.
\newblock {\em Mathematical Programming Glossary}.
\newblock INFORMS Computing Society, {\tt http://glossary.informs.org},
  2006--24.
\newblock Originally authored by Harvey J. Greenberg, 1999-2006.

\bibitem{sonnevend86}
G.~Sonnevend.
\newblock An ``analytical centre" for polyhedrons and new classes of global
  algorithms for linear (smooth, convex) programming.
\newblock In {\em System Modelling and Optimization. Lecture Notes in Control
  and Information Sciences}, volume~84, pages 866--875, Berlin, Heidelberg,
  1986. Springer, Springer Berlin Heidelberg.

\bibitem{sonnevend88}
G.~Sonnevend.
\newblock New algorithms in convex programming based on a notion of
  “centre”(for systems of analytic inequalities) and on rational
  extrapolation.
\newblock In {\em Trends in Mathematical Optimization: 4th {F}rench-{G}erman
  Conference on Optimization}, pages 311--326. Springer, 1988.

\bibitem{terlaky2009}
Tam{\'a}s Terlaky.
\newblock Twenty-five years of interior point methods.
\newblock In {\em Decision Technologies and Applications}, pages 1--33.
  INFORMS, 2009.

\bibitem{vavasis96}
S.~Vavasis and Y.~Ye.
\newblock A primal-dual interior point method whose running time depends only
  on the constraint matrix.
\newblock {\em Mathematical Programming}, 74(1):79--120, 1996.

\bibitem{wright98}
M.~Wright.
\newblock The interior-point revolution in constrained optimization.
\newblock In {\em High Performance Algorithms and Software in Nonlinear
  Optimization}, pages 359--381. Springer, 1998.

\bibitem{wright05}
M.~Wright.
\newblock The interior-point revolution in optimization: history, recent
  developments, and lasting consequences.
\newblock {\em Bulletin of the American Mathematical Society}, 42(1):39--56,
  2005.

\bibitem{wright1994}
Margaret~H Wright.
\newblock Some properties of the {H}essian of the logarithmic barrier function.
\newblock {\em mathematical Programming}, 67(1):265--295, 1994.

\bibitem{wright1999}
Stephen~J Wright.
\newblock Modified {C}holesky factorizations in interior-point algorithms for
  linear programming.
\newblock {\em SIAM Journal on Optimization}, 9(4):1159--1191, 1999.

\end{thebibliography}

\newpage

\section*{Supplement: Artistic Techniques and Lessons}

We present in this supplement the artistic techniques used to create the large 
piece of art in Figure~\ref{fig-bigWall}. The scale of this project challenged our 
intent to use what we had learned from smaller works to create an expansive piece 
for a spacious and well-traveled hallway. Techniques that had been successful
for moderate pieces just didn't scale easily, and we spent months tinkering 
with software, 3D printers, laser cutters, and dyeing processes before gaining 
the finesse to create a professional product. We catalog our methods here
to aid others interested in similar projects.

This work combined three components, those being wood cutouts of central paths for 
the front facing `frames,' 3D printed tilings to emboss patterns on the underlying 
paper, and the dyeing process to create the dynamic of color. We address each
of these processes below.

\subsection{Wood Cutouts}

We started with a collection of paths in a $k$-gon for each wood frame, but these paths
are mathematically one-dimensional and have no area, which means that they cannot be
used to create wood frames without imposing a width, i.e. we needed to `thicken' the paths. 
We saved each image as a scalable vector graphic (svg) file and used the open-source svg 
editor Inkscape to thicken the paths. The automated process in Inkscape unfortunately left 
the corners empty as seen in Figure~\ref{fig-noCor}. The light-blue curves are our computed 
central paths, and the larger dark-blue areas are Inkscape's thickened versions. The 
intersection of the thicker paths left voids at the corners that we had to manually correct
in each image, see Figure~\ref{fig-wCor}. We then used a laser cutter to cut wood frames 
that tiled without awkward gaps.

\begin{figure}[H]
\begin{center}
\begin{subfigure}[t]{0.45\linewidth}
\begin{center}
\includegraphics[width=0.8\linewidth]{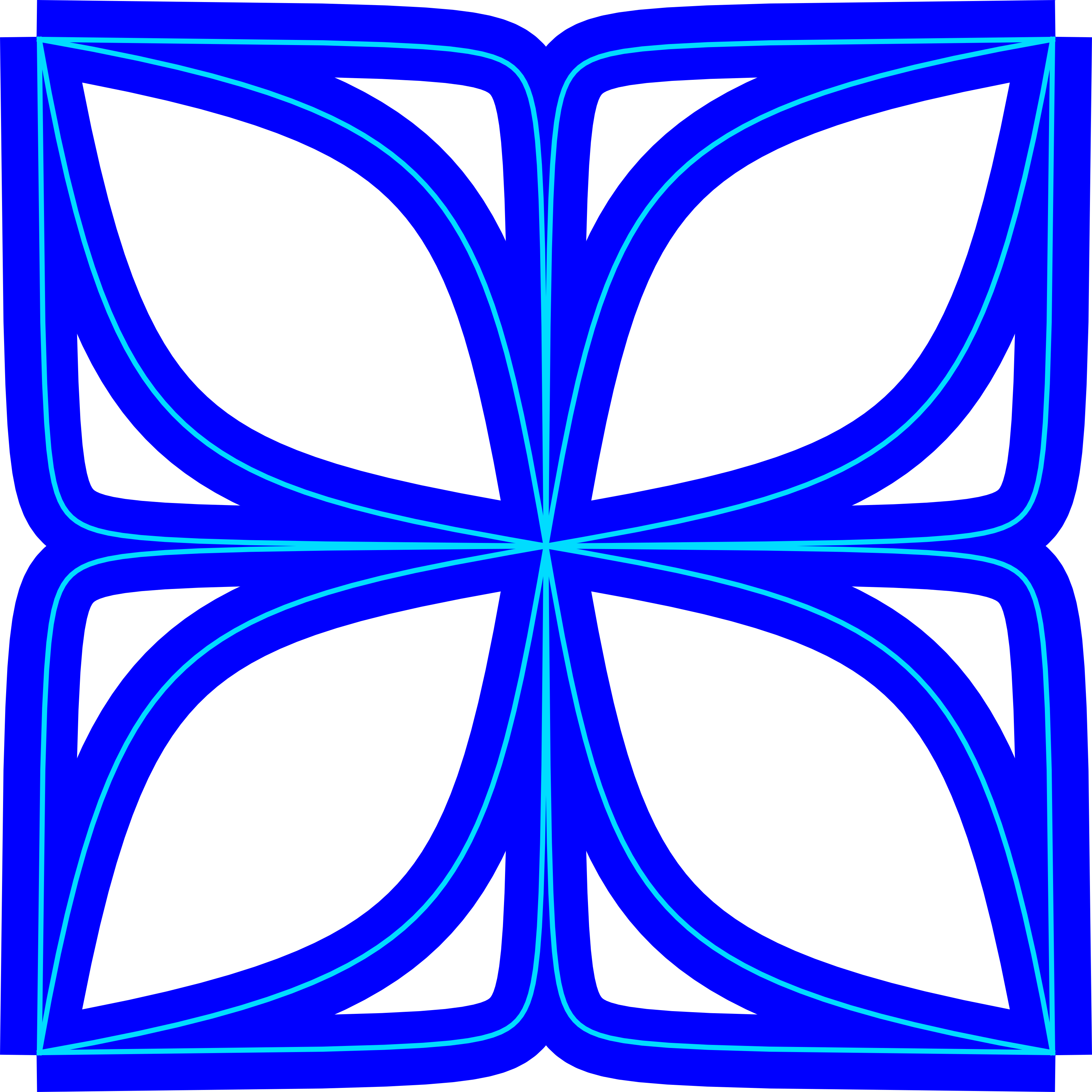}
\caption{\footnotesize Corners not filled in.} \label{fig-noCor}
\end{center}
\end{subfigure}
\hspace*{0.01\linewidth}
\begin{subfigure}[t]{0.45\linewidth}
\begin{center}
\includegraphics[width=0.8\linewidth]{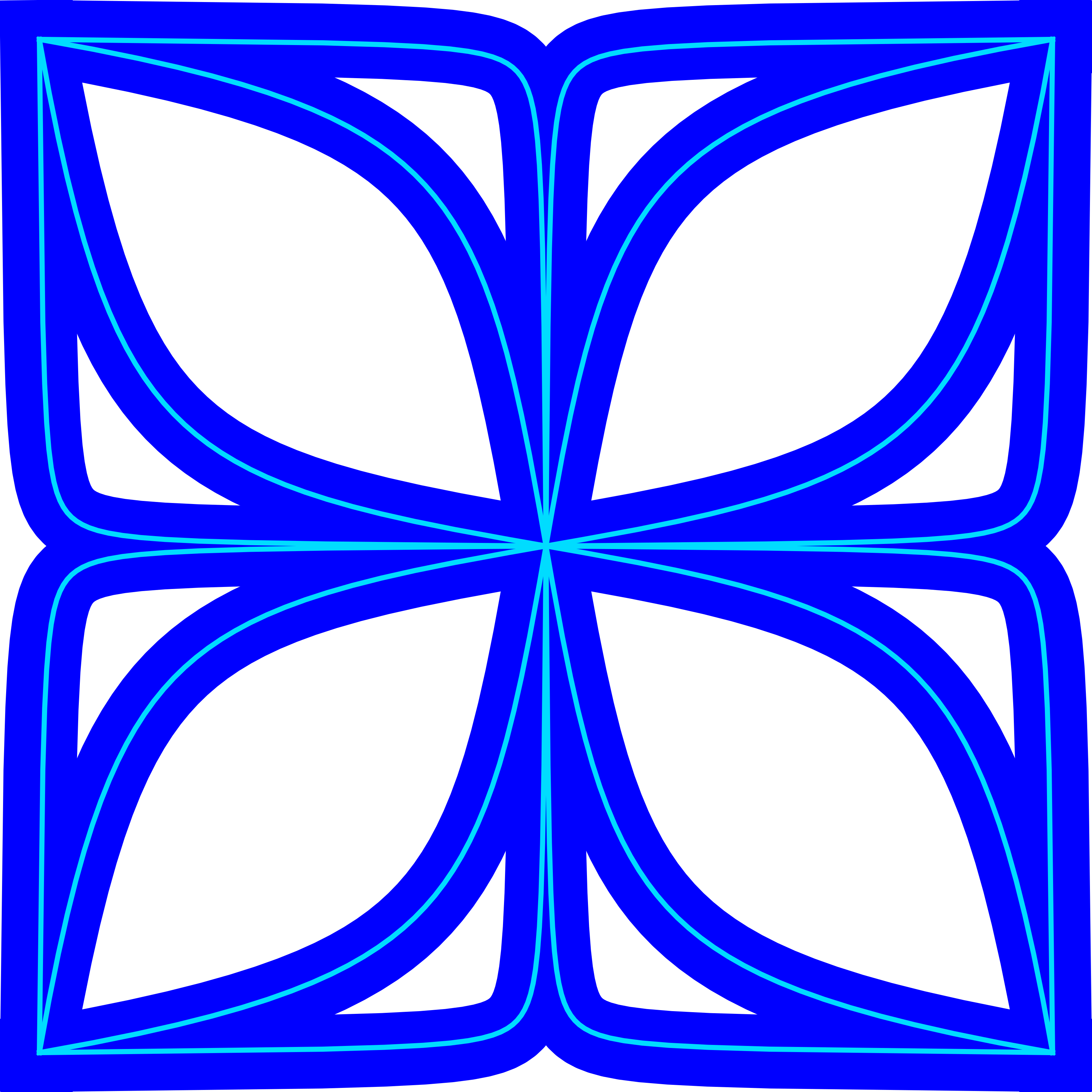}
\caption{\footnotesize Corners fixed.} \label{fig-wCor}
\end{center}
\end{subfigure}
\end{center}
\caption{Light blue original paths and dark blue thickened paths.}  \label{fig-cpr}
\end{figure}

\subsection{SolarFast Dyeing}

John Herschel invented solar/sun prints, called cyanotypes, in the 1840s,
which he produced by saturating a piece of paper or cloth with a mixture of 
ferric ammonium citrate and potassium ferricyanide and then exposing the saturated
item to UV light. The chemicals reacted under this exposure and dyed the item blue, 
with blueprints being an early example of the technique. A downside of 
cyanotyping is that we only have the single color of blue, but we fortunately
have a similar contemporary process that uses SolarFast that permits a variety 
of colors. We have experimented with both cyanotypes and SolarFast, and we 
generally prefer SolarFast.

\subsubsection*{3D Patterns}

The svg files of 2D patterns, like the ones that we used for laser cuttings, are 
not directly 3D printable because a thickened path in two-dimensions has no volume,
and 3D prints require surfaces of volumes. We have already noted our custom software
that extrudes circles along a path to create a printable stl file in Section~\ref{sec-3Dart}.
This code also works for paths in two-dimensions by setting the vertical $z$-coordinate
to zero, but it does so in a way that leaves the three-dimensional tessellations of the
thickened paths shy of being horizontally flat. We have instead used standard software in,
for instance, our printer's suite of utilities that accepts 2D svg images and thickens
them vertically for printing. This has proven trustworthy and has given us control of the
vertical height of the print, which is an important design element for the dyeing process.

The 3D printed patterns of the dyeing process are negatives because they shield
ultraviolet exposure and leave an undyed pattern. The general process is to
apply dye to watercolor paper, which is colorless at that point, place a 3D
printed pattern on the paper, cover the pattern with thick glass to hold everything
in place, expose the paper to sunlight to activate the chemicals and create
color, and then wash and dry the paper. This straightforward process works well
for smaller items, but it requires special attention as it scales for larger pieces.
The most daunting concern is that large pieces of watercolor paper tend to
warp and shrink as they are exposed, leaving the negative less than effective. The
flawed result has weak-looking paths that lack contrast and that tarnish the visual 
acuity of the underlying mathematics. We discovered solutions to several design 
decisions to remove, or at least limit, this concern.

The type of watercolor paper is important, and for larger pieces we found
that $140$ pound cold-pressed paper with a composition of $25\%$ cotton works well.
Heavier weights like $300$ pound paper can work nicely because they reduce
buckling; however, their heavier weight challenges gluing them to the wooden
frames even with reduced buckling. We experienced more buckling with lighter 
weight paper but were able to better press the print flat as we glued it to
the frame, creating a polished and taught appearance. Using a cold-pressed paper
gave us a slightly textured surface that worked well with SolarFast.

The negatives of our largest prints exceeded the dimensions of our FDM printer,
and as such, we printed the negatives in pieces that were then welded together
with a soldering iron to create one large pattern, see Figure~\ref{fig:neg}. 
The dimensions of the paths and the type of filament are both important for 
fitment and dyeing. We have used both PLA and PETG filaments, but we prefer PETG 
because it has thermal properties that allow the reuse of the pattern. PLA patterns
actually work a bit better for large prints because they conform better to the surface
of the paper, but they warm and deform during the exposure period, leaving them questionable
for subsequent prints. So PLA patterns work well but create waste and lengthen the 
creative timeline due to the need to print and assemble a new pattern for each print. The 
dimensions of the 3D printed paths are important because the width defines what the eye 
will see of a path, and the height alters how sunlight will interact with the dye. We 
have found that PETG with a path width varying from $1$ to $1.2$mm and a vertical height 
of $1.5$mm results in a negative that has accurate fitment and balanced flexibility, 
resulting in prints with good visual acuity and contrast.

\begin{figure}[h]
    \centering
    \includegraphics[width=0.4\linewidth]{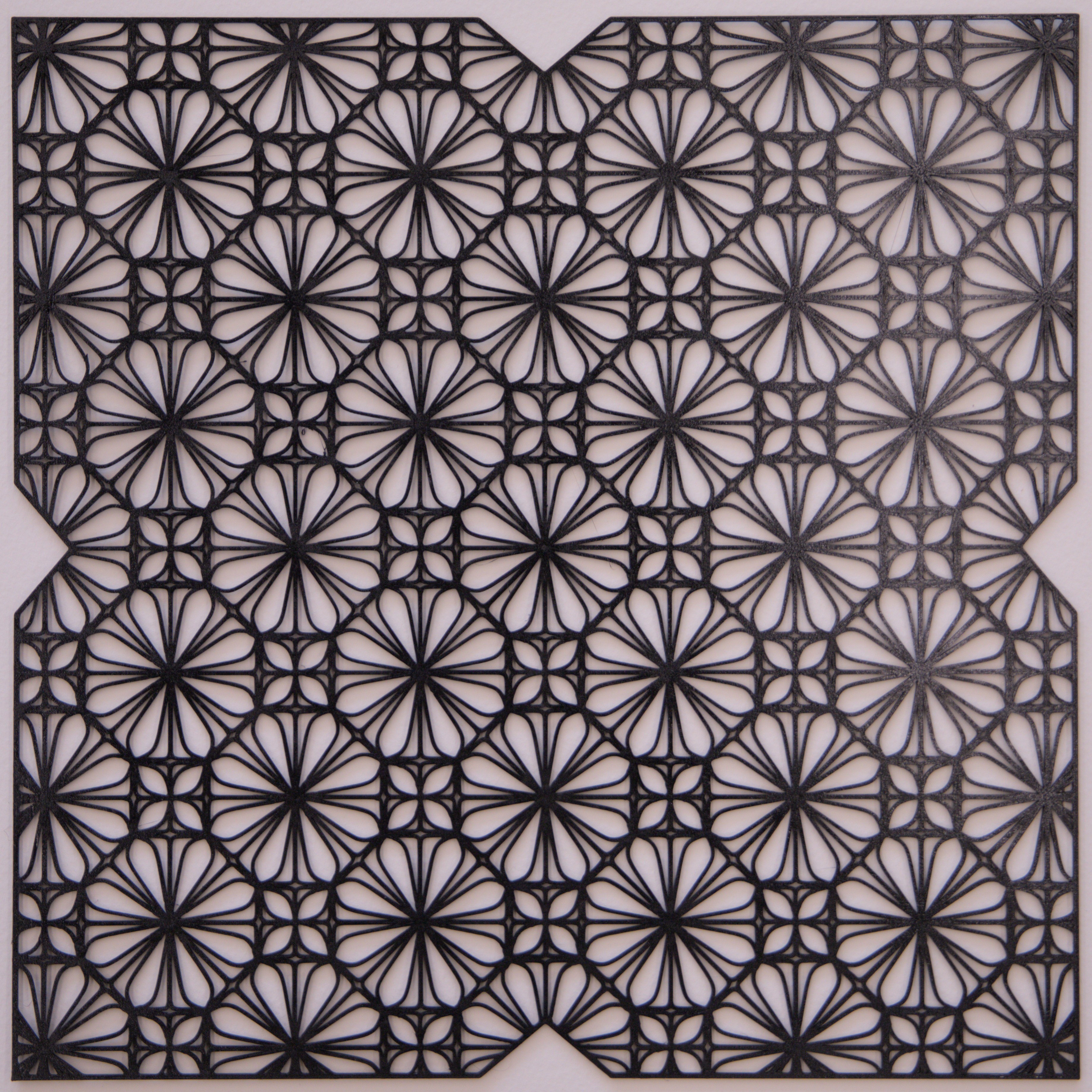}
    \caption{3D Printed Negative.}
    \label{fig:neg}
\end{figure}

\subsection*{Solar Printing}

The general printing process described above has several smaller but important steps
if the goal is to create high-quality large prints. The most important hindrance to
overcome is the fact that the paper adjusts as it warms during the exposure period.
We warn that we tried several ways of immobilizing the paper to limit the deleterious 
effect of exposing the paths that were supposed to shielded by the negative. None of
holding the paper taught with tape, adding additional glass for extra weight, or 
padding the backside of the paper to help impress the negative worked individually.
The tension created by the warming paper is impressive, and attempts to physically 
affix it proved ineffective. What instead worked was to delay the paper's
deformation by keeping it wet during the dyeing process. We accomplished this task
by wetting the paper prior to dyeing it, which left the paper sufficiently wet
during the exposure that it did not appreciably deform. This technique, along with
some physical binding, gave excellent results. Our standard procedure for large prints is below.
\begin{enumerate}
    \item Soak the paper in water for five to ten minutes.
    \item Combine dye and water in small bowls or containers in a roughly 2:1 ratio.
    \item Brush dye onto paper with a one-inch foam brush, with each color getting its own brush.
    \item Lightly mist the paper with water using a spray bottle.
    \item Use a two-inch round sponge to blot the paper, remove brush stokes, and 
        spread the color (in a quasi-random way).
    \item Place the paper on a flat padded surface.
    \item Place the 3D printed negative on top of the paper.
    \item Use your hands or a rubber brayer to slightly press the negative onto the wet paper,
        which somewhat adheres the negative to the paper.
    \item Place a piece of glass on top of the negative and press down firmly. 
        Multiple pieces of glass for extra weight can help.
    \item Take the piece outside and expose it in a shady spot for around twenty minutes - longer
        if the UV index is low. The print is done once you are happy with the color.
    \item Bring the print back inside and wash it following the SolarFast instructions. 
        Make sure the print does not become translucent while washing it. Stop washing 
        immediately if so.
    \item Place the washed print on a stretching board.
    \item Wipe the surface of the paper to remove any remaining dye with a dry paper towel.
    \item Tape the edges of the paper to the stretching board and wait for the print to dry.
\end{enumerate}
We comment that pre-soaking the paper affects the color of the dye, and most notably,
it results in the color being truer to the advertised hue. We suspect that this effect is
due to the extra dilution of the dye, but it could also be due in part because the dye
absorbs deeper into the paper. In any event, the concentration of the dye is an important
consideration of an artist as they use this technique.

\end{document}